\documentclass[11pt,leqno]{amsart}
\usepackage{amsmath,epsfig,graphicx,color, float,amssymb,mathrsfs}
\usepackage{subcaption}
\usepackage{relsize}
\usepackage{comment}
\usepackage[colorlinks, linkcolor=blue,citecolor=blue]{hyperref}
\usepackage{pdfpages}
\usepackage{graphicx}
\usepackage{mwe}
\usepackage{algorithm}
\usepackage{algorithmicx}
\usepackage{algpseudocode}
\usepackage{natbib}
\usepackage{dsfont}
\usepackage{tikz}

\usepackage{mathbbol}

\DeclareSymbolFontAlphabet{\amsmathbb}{AMSb}%

\usetikzlibrary{arrows,positioning,calc} 
\usetikzlibrary{decorations.pathreplacing}
\tikzset{
    >=stealth',
    punkt/.style={
           rectangle,
           rounded corners,
           draw=black, thick,
           text width=5.5em,
           minimum height=2em,
           text centered},
    punktl/.style={
           rectangle,
           rounded corners,
           draw=black, thick,
           text width=7em,
           minimum height=2em,
           text centered},
    pil/.style={
           ->,
           shorten <=4pt,
       shorten >=4pt
    },
    pildotted/.style={
           ->,
           shorten <=4pt,
           shorten >=4pt,
  dotted,
  },
    punktf/.style={
           rectangle,
           text width=4.0em,
           minimum height=1.5em,
           text centered},
    punktfTop/.style={
           rectangle,
           text width=4.0em,
           minimum height=1.5em,
           text centered,
           append after command={
               [thick,shorten >=0.2bp, shorten <=0.2bp]
               (\tikzlastnode.north west)edge(\tikzlastnode.north east)
}
    },
    punktfBot/.style={
           rectangle,
           text width=4.0em,
           minimum height=1.5em,
           text centered,
           append after command={
               [thick,shorten >=0.2bp, shorten <=0.2bp]
               (\tikzlastnode.south west)edge(\tikzlastnode.south east)
            }
    }
}

\usepackage{pgfplots}
\pgfplotsset{compat=1.17}
\usepackage{pgfplotstable}

\usepackage[margin=1.3in]{geometry}

\numberwithin{table}{section}
\numberwithin{figure}{section}
\numberwithin{equation}{section}

\definecolor{darkblue}{rgb}{.2, 0.2,.8}
\definecolor{darkgreen}{rgb}{0,0.5,0.3}
\definecolor{darkred}{rgb}{.8, .1,.1}

\newcommand{\dd}{\mathrm{d}}

\newcommand{\1}{{\mathds{1}}}

\newcommand{\N}{\amsmathbb{N}}
\newcommand{\R}{\amsmathbb{R}}
\newcommand{\E}{\amsmathbb{E}}
\renewcommand{\P }{{\amsmathbb P}}

\newtheorem{lemma}{Lemma}[section]
\newtheorem{theorem}[lemma]{Theorem}

\newtheorem{corollary}[lemma]{Corollary}

\newtheorem{remark}{Remark}[section]

\usepackage{bm}

\newcommand{\ap}[1]{\mathbb{#1}}
\newcommand{\tld}[1]{\widetilde{#1}}
\newcommand{\apt}[1]{\tld{\mathbb{#1}}}

\allowdisplaybreaks

\begin{document}
\title[Approximations of semi-Markov processes and insurance policy valuation]{Approximations of semi-Markov processes and insurance policy valuation}

\author[M. Bladt]{Martin Bladt}
\thanks{Martin Bladt. University of Copenhagen, Department of Mathematical Sciences, Universitetsparken 5,
2100 K{\o}benhavn, Denmark, martinbladt@math.ku.dk}

\author[A. Minca]{Andreea Minca}
\thanks{Andreea Minca. Cornell University, School of Operations Research and Information Engineering, Ithaca,
NY, 14850, USA, acm299@cornell.edu}

\author[O. Peralta]{Oscar Peralta}
\thanks{Oscar Peralta. Cornell University, School of Operations Research and Information Engineering, Ithaca, NY, 14850,
op65@cornell.edu}

\maketitle

\begin{abstract}

Inspired by a duration-dependent life insurance model, we consider continuous-time semi-Markov jump processes, { initially assumed to have a finite state-space}. We develop approximations using jump processes that are time-homogeneous Markov, conditioned on a high-intensity Poissonian grid (grid-conditional). Our results are based on a recent adaptation of the uniformization principle, which yields a strongly pathwise convergent sequence of jump processes. In contrast to traditional methods that use classical approximations to integro-differential equation solutions to compute value functions, our approximations result in easily implementable expressions, making them valuable in situations where evaluating pathwise distributional functionals for the original semi-Markov process is challenging. Our homogeneous approximation, initially grid-conditional, evolves into an unconditional version that remains effective under reasonable regularity assumptions. { We then relax the finite state-space assumption and show how our results can be extended to a general measurable state-space}. We illustrate the practicality of our approach with a disability insurance model, using realistic underlying semi-Markov process parameters.
\end{abstract}
{\bf Keywords:} Disability insurance;   pathwise approximation;  policy valuation; semi-Markov models;.

{\bf Mathematics Subject Classification:} 60K15, 91G05, 91G60.

\section{Introduction}

The traditional mathematical treatment of life insurance is centered around mortality tables and discrete-time computations, with virtually no connection to probability theory or continuous-time models. Around 1875, the Danish astronomer and mathematician Thorvald Thiele introduced the differential equation that governs the net premium reserve of a basic life insurance model, marking the first instance that a continuous-time life insurance framework was formally considered. Many decades later, \cite{sverdrup1952basic,hickman1964statistical} suggested formal probabilistic approaches for some special cases, while the full continuous-time finite-state Markov jump-process framework was first systematically considered in \cite{hoem1969markov}.

Since the formulation of the basic Markov model, and compounded by the consequential fact that a substantial number of industries in northern European countries have by now adopted continuous-time models, there has been a continuing theoretical and practical demand for extending such models to more flexible stochastic frameworks. Some notable cases have been stochastic mortality rate models (cf. \cite{cairns2008modelling}), the integration of financial mathematics and life-insurance concepts (cf. \cite{steffensen2006surplus}), and the introduction of a semi-Markov framework (cf. \cite{hoem1972inhomogeneous}, see also a modern treatment in \cite{buchardt2015cash}). The literature in these directions is vast and continues to expand. 
{ 
}
In this paper, we focus on the semi-Markov model, a duration-dependent stochastic process, within the context of general life insurance models featuring duration-dependent payments. Semi-Markov processes, first studied in their time-homogeneous version, represent jump processes where interarrival times are not exponentially-distributed, unlike time-homogeneous Markov processes. These are also known as Markov renewal processes in the literature. The seminal works of \cite{levy1954processus}, \cite{smith1955regenerative}, and \cite{cinlar1969markov} laid out the foundational groundwork for these models. Further, in a life-insurance framework, a time-inhomogeneous version of these processes which considers the jump behavior dependent on both time and the elapsed time since the last jump was explored in \cite{hoem1972inhomogeneous,janssen1984finite}. { Alternatively, employing the theory of piecewise deterministic Markov processes (e.g. \cite[Chapter 2]{davis2018markov} and \cite[Chapter 7]{jacobsen2006point}), semi-Markov models can be interpreted as intensity-driven jump processes with dynamics embedded in a general state-space. This perspective allows for a more general representation of states in the semi-Markov model, extending beyond the discrete state-space assumed in \cite{hoem1972inhomogeneous, janssen1984finite}.}

The limitations of simple Markovian models in accurately capturing life insurance products are well documented in both practice and literature, as seen in \cite{janssen1966application,hoem1972inhomogeneous,buchardt2015cash}. While some special cases permit clearer mathematical treatments similar to the Markov case, the general scenario remains methodologically complex and numerically delicate. This complexity arises from the need to determine both the transition rates and probabilities of the underlying process, which, in non-trivial multi-state structures, often requires solving Thiele's or Kolmogorov's differential equations. { In the actuarial literature, traditional approaches to modeling prospective reserves in life insurance have utilized the so-called forward and backward differential methods. The forward method, based on Kolmogorov’s forward differential equations, calculates transition probabilities and expected cash flows to determine reserves at a specific point in time, whereas the backward method uses Thiele’s differential equations to compute reserves across all time points. We refer to \cite[Section V.3]{asmussen2020risk} and \cite[Section 1.1]{furrer2020multi} for more details on the advantages and disadvantages of each method. As an alternative, simulation-based methods, such as the efficient Monte Carlo scheme proposed in \cite{fox1986discrete}, can be employed, although they are generally applicable only to time-homogeneous semi-Markov processes. For brevity, we will refer to the time-inhomogeneous semi-Markov model simply as ``semi-Markov'' from here on.}

{ Here, we propose a novel and general-purpose approximation that avoids the need to set up differential equations, establishing itself as an alternative to both forward and backward methods, while still aligning with the objectives of the forward method. Specifically, under notably robust conditions, we systematically construct a tractable process whose paths converge to those of the semi-Markov model in question. Unlike traditional time-discretization schemes, our approach operates in a continuous-time setting and avoids the limitations associated with solving differential equations. This allows for probabilistic convergence guarantees and offers a more accurate representation of the stochastic behavior of the original process. For the simpler case of time-inhomogeneous Markov processes, a construction of such an approximation was presented in \cite{bladt_peralta}. We now explain the general framework for the semi-Markov case.}

Heuristically, the proposed approximation constructs the semi-Markov model using uniformization (cf. \cite{jensen1953markoff}; see also \cite{van1992uniformization} for the inhomogeneous case) on top of a high-intensity Poisson process. This approach can be viewed as a pathwise strong convergence version of the Uniform Acceleration method \cite{massey1998uniform}, generalized to semi-Markov models. The core idea is to replicate the exact same sequence of states using another Poisson process that is identically distributed and independent of the original one. Conditional on the arrival times of the original Poisson process—referred to as \emph{grid-conditional}—this approximation becomes a simple time-homogeneous Markov jump process with transition probabilities that have an algorithmically tractable closed expression. Furthermore, under additional continuity assumptions on the underlying jump intensities, we present an \emph{unconditional} time-homogeneous Markov jump process with transition probabilities that converge to those of the original semi-Markov model.

As a particular consequence of our convergence results, we obtain a valuation approximation with explicit components that is guaranteed to approximate the target value function to arbitrary precision. Another implicit consequence is the generalization of absorption-time approximations as given in Section 4 of \cite{ahmad2023estimating}.
{ Moreover, we provide an explicit solution to a conjecture recently posed in \cite{ahmad2023aggregate} regarding the denseness of aggregate Markov models in relation to semi-Markov models.}

{ In contrast to methods that rely on discretizing time and duration---resulting in a countably infinite system of ODEs, as seen in \cite{buchardt2015cash} and \cite{adekambi2017integral}---our approach approximates the semi-Markov process in a pathwise sense. This strategy avoids the numerical complexities of solving integro-differential equations or large systems of ODEs, while leading to a  transparent understanding of the approximating process. 
} Our construction can be regarded as a general result in applied probability which may be employed in any setting where approximating value functions (or other functionals) of a semi-Markov process is of interest. In particular, we believe that continuous-time life insurance mathematics is an application well-suited to demonstrate our results in an elegant and practically useful manner. { However, it is important to note that our findings are applicable only in the smooth case and do not address the non-smooth scenarios considered by \cite{helwich2008durational}}. Several extensions emerge from our study. First, in Section~\ref{sec:credit} we suggest how to apply our framework to reduced form credit risk modeling, where the states can represent ratings of a company and we seek to evaluate a defaultable bond. Secondly, in Section~\ref{sec:control}, we show that under mild conditions, the approximation framework also holds for controlled semi-Markov processes, broadening the scope and applicability of our methodology beyond life insurance product valuation. We also suggest how to use our approximation for value functions in optimal control problems.

The remainder of the paper is structured as follows. In Section \ref{sec:preliminaries} we introduce the life insurance setup, and the underlying semi-Markov process.
In Section \ref{sec:approxgen1} we introduce our general grid conditional approximation. In Section \ref{sec:unconditional}, we provide unconditional approximations, under entry-wise Lipschitzness of the intensity matrix.
Section \ref{sec:examples} applies the framework to life insurance policy evaluation and presents extensions to the valuation of credit risk sensitive securities. It illustrates a potential use-case to optimal control. Technical proofs may be found in Appendices \ref{sec:strongSMJP} and \ref{sec:convergence_densities}.

\section{Preliminaries}
\label{sec:preliminaries}
This section provides some of the relevant background in the continuous-time life insurance formulation, and on the definition and construction of semi-Markov processes. We return to the former once the main strong approximation is established, while the latter motivates the construction principle of the approximating sequence.

\subsection{Life insurance setup}\label{sec:setup}
This section introduces the standard setup for a multi-state life insurance model. { We refer readers to \cite{norberg1991reserves,milbrodt1997markov,milbrodt1999hattendorff,milbrodt2000hattendorff,buchardt2015life} for a comprehensive review of such models in actuarial science.} Here, we present the semi-Markov model originally introduced in \cite{hoem1972inhomogeneous} and also discussed with modern notation in \cite{buchardt2015cash}.

Consider an insurance policy issued at time $0$ with a terminating time $T\in (0,\infty)$, beyond which there is no more coverage. We deal with a c\`adl\`ag stochastic process, defined on a probability space $(\Omega, \mathcal{F}, \mathds{P})$, taking values in a finite set $\mathcal{J}=\{1,\dots,J\}$. This set corresponds to different states of the policy. The stochastic process denoting the state of the policy at any given time is given by $Z=\{Z(t)\}_{t\ge 0}$
and is commonly referred to as the multi-state model. We define 
\begin{equation}\label{eq:U_def1}U(t)=\inf\{r>0: Z(t-r)\neq Z(t)\},\end{equation}
as the elapsed time since the last jump. We will refer to $U=\{U(t)\}_{t\ge 0}$ as the duration process. In this paper, we are interested in multi-state models $Z$ such that the bivariate process $(Z,U)$ is Markovian, known as semi-Markov processes. We will develop a deeper understanding of semi-Markov models in Subsection \ref{sec:SMJP}; for now, its definition will suffice.

When considering the value of a policy at any time other than inception, an updated amount will depend on the information available at time $t\in[0,T]$, including the duration process. {We denote $\mathcal{F}_t$ the $\mathds{P}$-completion of $\sigma(Z_s;\,s\le t)$.} Furthermore, we define the $J$-dimensional counting process $N=(N^1,\dots, N^J)$ as 
\[ N^k(t)=\sum_{s\in[0,t]} \1\{Z(s)=k,Z(s-)\neq k\},\quad k\in\mathcal{J}, \]
which counts the number of jumps into state $k$ experienced up to time $t$. {Here, we define $Z(0-)=Z(0)$ for convenience.}

Assume that $B^{k,u}(t)$ denotes, for each $k\in\mathcal{J}$ and $u\ge0$, a deterministic payment rate process for when $Z(t)$ is in state $k$. We decompose it into its absolutely continuous and discrete parts by
\[ B^{k,u}(t)=\int_0^t b^{k,u}(s)\dd s +\sum_{s\in[0,t]}\Delta B^{k,u}(s), \]
where $\Delta B^{k,u}(s)=B^{k,u}(s)-B^{k,u}(s-)$. Similarly, for payments occurring during state transitions, we define $b^{(j,u),(k,v)}(t)$ as the lump-sum payment occurring when transitioning from state $j$ with duration $u$ to state $k$ with duration $v$ at time $t$. Note that the only possible transitions occur for $v=0$.

Combining the above components, the total amount of benefits minus premiums is defined as the stochastic process
\[ \begin{split}
\dd B(t)&= \dd B^{Z(t),U(t)} + \sum_{(k,u)\neq (Z(t-),U(t-))}b^{(Z(t-),U(t-)),(k,u)}(t) \dd N^{k}(t)\\
&=\dd B^{Z(t),U(t)} + \sum_{k \neq Z(t-)}b^{(Z(t-),U(t-)),(k,0)}(t) \dd N^{k}(t).
\end{split} \]

Introducing the concept of return on investment for the portfolio and denoting an associated nonnegative deterministic interest rate $r(t)$ for $t\ge 0$, we may become interested in the discounted process 
\[ \int_t^T e^{-\int_t^s r}\dd B(s), \]
for the valuation of the insurance policy. Here, we use the shorthand notation $\int_t^s r=\int_t^s r(u)\dd u$. For instance, its first conditional moment given the information at time $t$ is called the {\textit{prospective reserve}}:
\[ V(t)=\E\left[\int_t^T e^{-\int_t^s r}\dd B(s) \mid \mathcal{F}_t\right], \]
while higher-order moments are usually of significant but secondary importance. The function $V$ represents the mean value of all future payments {up to time $T$}, given the information available at time $t$. Typically, most payment strategies are pre-set, and the condition {$V(0)=0$ (at inception, all total future payments should, on average, cancel out) is used to determine the last free coefficient to produce an \textit{actuarially fair} contract. In  case of an initial premium $b_0\le 0$, one instead has the more general expression $V(0^-) := V(0) + b_0 = 0$.}

By the semi-Markov property, we have
\[ V(t)=\E\left[\int_t^T e^{-\int_t^s r}\dd B(s) \mid Z(t),U(t)\right], \]
and so we now define the function of interest:
\[ V(t;i,u)=\E\left[\int_t^T e^{-\int_t^s r}\dd B(s) \mid Z(t)=i,U(t)=u\right],\quad i\in\mathcal{J},\: u\ge0. \]

{ Finally, we point out that the main focus of an actuary is often not directly $\ap{V}_{i}$, but rather the following quantity referred to as the \textit{expected cashflow}:
\begin{align}\label{cashflow_definition}
{ \mathcal{C}_{t;i,u}(\dd s) = \E\left[\dd B(s) \mid Z(t) = i, U(t) = u\right],\quad i\in\mathcal{J},\: u\ge0.}
\end{align}
This expected cashflow can be used to construct the value function $\ap{V}_{i}$. As the name indicates, the cashflow represents the  instantaneous (and un-discounted) rate of payment transfer from the insurer to the insured. Premium payments may be included in this transfer, accounted for as negative payments to the insured.}

\subsection{Semi-Markov jump processes}\label{sec:SMJP}

Let $Z=\{Z(t)\}_{t\ge 0}$ be a c\`adl\`ag stochastic process, defined on the probability space $(\Omega,\mathcal{F},\P)$, with state space $\mathcal{J}=\{1,2,\dots, J\}$. We say that $Z$ is a semi-Markov process if for all $0<t<s$ and $j\in\mathcal{J}$,
\begin{align*}
\P(Z(s)=j\,|\,Z(r), 0\le r\le t) =\P(Z(s)=j\,|\,Z(t), U(t)),
\end{align*} 
where { $U(t)$ corresponds to the duration process as defined in \eqref{eq:U_def1}.}

{ In this work, we focus on the class of time-inhomogeneous semi-Markov processes driven by a family of intensity matrices. We follow the description of semi-Markov processes from \cite{janssen1984finite}, which aligns well with the purposes of our study. However, as noted in the Introduction, we remark that alternative constructions can be made using intensity-driven Markov jump processes, such as those described in \cite[Chapter 2]{davis2018markov} and \cite[Chapter 7]{jacobsen2006point}. Here, the intensity matrices $\{\bm{\Lambda}(s,v)\}_{s,v\ge 0}$ (where $\bm{\Lambda}(s,v)=\{\Lambda_{ij}(s,v)\}_{i,j\in\mathcal{J}}$) possess the following characteristics:}
\begin{itemize}
       \item For all $v\ge 0$, the mapping $s\mapsto \bm{\Lambda}(s, v)$ is c\`adl\`ag,
       \item For all $s\ge 0$, the mapping $v\mapsto \bm{\Lambda}(s, v)$ is c\`adl\`ag,
       \item There exists $\gamma_0\le \infty$ such that $\sup_{s,v\ge 0, i\in\mathcal{J}} |\Lambda_{ii}(s,v)|\le \gamma_0$.
\end{itemize}
Under these conditions, we say that the semi-Markov process $Z$ is driven by $\{\bm{\Lambda}(s,v)\}_{s,v\ge 0}$ if, for all $i,j\in\mathcal{J},\, s,v\ge 0$,
\begin{align}
\mathds{P}\left(Z(s+h)=j\,|\, Z(s)=i, U(s)=v\right) = \delta_{ij} + \Lambda_{ij}(s,v)h + o(h),\label{eq:Z_semi1}
\end{align}
where $\delta_{ij}$ denotes the Kronecker delta function, and $o(h)$ represents a generic function $f:\R_+\rightarrow\R$ such that $\lim_{h\downarrow 0} f(h)/h = 0$.

We now briefly demonstrate how, for any family of intensity matrices $\{\bm{\Lambda}(s,v)\}_{s,v\ge 0}$ with the aforementioned characteristics, we can construct its associated semi-Markov process $Z$ on the interval $[t,\infty)$ (conditional on $Z(t)=i, U(t)=u$) using uniformization. 

First, consider the case $t=0$ and $U(t)=0$; the general case will follow from time-shifting arguments, which we detail at the end of this section. Here, assume that $(\Omega,\mathcal{F},\P)$ supports:
\begin{itemize}
       \item a Poisson process $\Gamma_0=\{\Gamma_0(s)\}_{s\ge 0}$ with intensity $\gamma_0$,
       \item a sequence $\{Y_\ell\}_{\ell\ge 1}$ of independent $\mbox{Unif}(0,1)$ random variables,
\end{itemize}
which are mutually independent.  Denote by $T_0=0, T_1, T_2,\dots$ the arrival times associated with $\Gamma_0$, and recursively construct the discrete-time bivariate process $\{(\zeta_\ell, \upsilon_\ell)\}_{\ell \ge 0}$ where $\zeta_0=i$, $\upsilon_0=0$, and for $\ell\ge 0$

\begin{align*}
\zeta_{\ell+1} = j\quad&\text{if}\quad Y_{\ell+1}\in \left[\sum_{k=0}^{j-1} P_\ell (k), \sum_{k=0}^{j} P_\ell(k)\right)\\&\text{with}\quad P_\ell(k)=\delta_{\zeta_\ell, k} + \frac{\Lambda_{\zeta_\ell, k}(T_{\ell+1}, \upsilon_{\ell} + (T_{\ell+1}-T_\ell))}{\gamma_0},\end{align*}
\begin{align*}
\upsilon_{\ell+1} = \left\{\begin{array}{ccc} \upsilon_{\ell} + (T_{\ell+1}-T_\ell) & \text{if} & \zeta_{\ell+1}=\zeta_{\ell},\\
0 & \text{if} & \zeta_{\ell+1}\neq\zeta_{\ell}. \end{array}\right.
\end{align*}
In the scheme mentioned above, given $\zeta_\ell, \upsilon_{\ell}, T_0, T_1, T_2,\dots$, the weights $P_{\ell}(k)$ represent the probability that $\zeta_{\ell+1}$ is equal to $k$. Simultaneously, we update the time elapsed since the last jump by time $T_{\ell+1}$, denoted by $v_{\ell +1}$. Next, we define $Z=\{Z(s)\}_{s\ge 0}$ as
\begin{equation}\label{eq:Zunif1}Z(s)= \zeta_{\Gamma_0(s)},\quad s\ge 0.\end{equation}
It can be shown that with $U(s)$ defined as in (\ref{eq:U_def1}),
\begin{equation}\label{eq:U2}U(s)=\upsilon_{\Gamma_0(s)} + (s-T_{\Gamma_0(s)}).\end{equation}
Using (\ref{eq:U2}) and considering the right-continuity of $\bm{\Lambda}(\cdot,\cdot)$, for a small $h>0$,
\begin{align*} 
\mathds{P}&\left(Z(s+h)=j\,|\, Z(s)=i, U(s)=v\right)\\
& = \mathds{P}\left( \Gamma_0(r)=\Gamma_0(s)\text{ for all }r\in(s,s+h] \,|\, Z(s)=i, U(s)=v\right)\\
&\qquad \times \mathds{P}\left(Z(s+h)=j \,|\, \Gamma_0(r)=\Gamma_0(s)\text{ for all }r\in(s,s+h], Z(s)=i, U(s)=v\right)\\
& + \mathds{P}\left( \Gamma_0(r)\neq \Gamma_0(s)\text{ for some }r\in(s,s+h]  \,|\,  Z(s)=i, U(s)=v\right) \\
&\qquad \times \mathds{P}\left(Z(s+h)=j\,|\, \Gamma_0(r)\neq \Gamma_0(s)\text{ for some }r\in(s,s+h],  Z(s)=i, U(s)=v\right) \\
& = (1-\gamma_0 h + o(h))\delta_{ij} + (\gamma_0h + o(h))\left(\delta_{ij} + \frac{\Lambda_{ij}(s, v)}{\gamma_0}\right) = \delta_{ij} + \Lambda_{ij}(s, v) h + o(h),
\end{align*}
which confirms the condition (\ref{eq:Z_semi1}), ensuring that $Z$ is a semi-Markov process driven by $\{\bm{\Lambda}(s,v)\}_{s,v\ge 0}$.

For a fixed $t\ge 0$, to construct the bivariate process $(Z,U)$ on $[t,\infty)$ conditional on $Z(t)=i$ and $U(t)=u$, first create a semi-Markov process $Z^*$ and its associated duration process $U^*$ with the following properties. $Z^*$ operates in the state space $\{0,1\}\times \mathcal{J}$, with $Z^*(0)=(0,i)$, and is driven by the family of intensity matrices $\{\bm{\Lambda}^*(s,v)\}_{s,v\ge 0}$, where 
\begin{align}\label{state_augmentation}
\bm{\Lambda}^*(s,v)  =\begin{pmatrix} \bm{\Delta}(t+s,u+v) & \bm{\Lambda}(t+s,u+v) - \bm{\Delta}(t+s,u+v)\\ \bm{0} & \bm{\Lambda}(t+s,v)\end{pmatrix},
\end{align}
\begin{align*}
\bm{\Delta}(t+s,u+v)=\begin{pmatrix}\Lambda_{11}(t+s, u+ v) &  & &\\ & \Lambda_{22}(t+s, u+ v)  & &\\ & &\ddots & \\ &&&\Lambda_{JJ}(t+s, u+ v)
\end{pmatrix}.\end{align*}
For all $s\ge t$, we set $Z(s)=\pi_2(Z^*(s-t))$, where $\pi_2: \{0,1\}\times \mathcal{J}\rightarrow \mathcal{J}$ indicates the second coordinate projection function, and define
\begin{align*}
U(s)= \left\{ \begin{array}{ccc} u+ U^*(s-t)& \text{if}& U^*(r) > 0 \text{ for all } 0<r\le s-t\\ U^*(s-t)& \text{if}& U^*(r) = 0 \text{ for some } 0<r\le s-t.\end{array}\right.
\end{align*}
For this process, for $s\ge t$ and $j\neq i$,
\begin{align*} 
\mathds{P}&\left(Z(s+h)=j\,|\, Z(s)=i, U(s)=v\right)\\
& = \Lambda_{ij}(s, v)h + o(h),
\end{align*}
validating that the process $Z$ is semi-Markov with the desired properties. 
Given these considerations, in subsequent sections, we will focus on constructions over $[0,\infty)$ with $U(0)=0$ only, understanding that other scenarios can be addressed using the shifting method described above.

\section{Grid-conditional approximations of semi-Markov processes}\label{sec:approxgen1}

This section provides the construction of the main strong approximation of semi-Markov processes, which is done conditionally. We explain in detail the construction, while we delegate the proof of strong convergence in the Skorokhod $J_1$-topology to Appendix \ref{sec:strongSMJP}. We also provide a scheme for efficient computation of the transition probabilities of the approximation, and illustrate how these drastically facilitate the evaluation of the value function in the life insurance setup.

\subsection{Construction of strongly convergent scheme}\label{sec:approx1}
We now construct a pathwise approximation of the process $Z$ defined via (\ref{eq:Zunif1}) using approximate uniformization.

Fix some $\gamma\ge \gamma_0$. Assume that the probability space $(\Omega, \mathcal{F}, \P)$ supports an additional auxiliary independent Poisson process $\Gamma'$ with intensity $\gamma-\gamma_0$. Let $\Gamma$ be the superposition of $\Gamma_0$ and $\Gamma'$, making $\Gamma$ a Poisson process with intensity $\gamma$. Let $\ap{\Gamma}$ be an identical and independent copy of $\Gamma$, also supported by $(\Omega, \mathcal{F}, \P)$. Define $\{\chi_\ell\}_{\ell\ge 0}$ as the arrival times of $\Gamma$ and the sequence $\{\eta_\ell\}_{\ell\ge 0}$ by
\[\eta_\ell = Z(\chi_\ell),\quad \ell\ge 0.\]
Note that the process $Z$ can be reconstructed from $\{\eta_\ell\}_{\ell\ge 0}$ and $\Gamma$ via the equality
\begin{equation}\label{eq:Zunif2} 
Z(s) = \eta_{\Gamma(s)},\quad s\ge 0.
\end{equation}
For our approximation $\ap{Z}$ to $Z$, the key idea is to replace $\Gamma$ with $\ap{\Gamma}$ in the r.h.s. of (\ref{eq:Zunif2}), that is, let 
\begin{equation}\label{eq:Zdag-unif2} 
\ap{Z}(s) = \eta_{\ap{\Gamma}(s)},\quad s\ge 0.
\end{equation}

In essence, the process $\ap{Z}$ visits exactly the same states as $Z$ does, but does so using a ``different clock''; see {Figure \ref{fig:diffclocks1}}
\begin{figure}[!htbp]\centering
\includegraphics[width=0.69\textwidth]{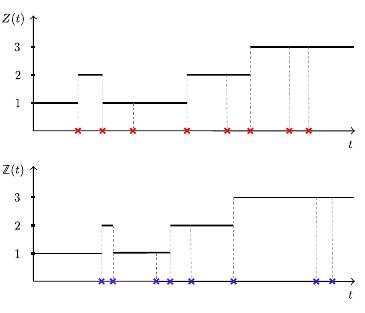}
\caption{Paths of the processes $Z$ and $\ap{Z}$ traversing the state space $\mathcal{J}=\{1,2,3\}$. Poissonian times of $\Gamma$ associated to $Z$ are shown with a red cross, while those of $\ap{\Gamma}$ associated to $\ap{Z}$ are shown with a blue cross. Note that $Z$ and $\ap{Z}$ sequentially visit the same states at the Poissonian epochs $\Gamma$ and $\ap{\Gamma}$.}\label{fig:diffclocks1}
\end{figure}
However, since $\Gamma$ and $\ap{\Gamma}$ have identical distributional characteristics, we can expect their paths to be \emph{close} to each other: we will investigate this statement in a precise way later in Theorem \ref{th:strong_simp}. Meanwhile, let us compute the distributional characteristics of the approximated process $\ap{Z}$ when conditioned on $\Gamma$. To do so, we first need to define an additional auxiliary process, $\ap{M}$, by letting
\[\beta_0=0, \quad \beta_{\ell+1} =\left\{\begin{array}{ccc} \beta_{\ell} + 1 &\mbox{if}& U(\chi_{\ell+1})\neq 0,\\ 0 &\mbox{if}& U(\chi_{\ell+1}) = 0, \end{array}\right.\]
\[\mbox{and} \quad \ap{M}(s)=\beta_{\ap{\Gamma}(s)}.\]
In short, $\ap{M}$ is a process that jumps at the same epochs as $\ap{\Gamma}$, either increasing by $1$ or restarting at $0$.
\begin{theorem}\label{th:apz-matrix-exp}
The process $(\ap{\Gamma},\ap{M},\ap{Z})$ with state space $\N_0\times \N_0\times \mathcal{J}$, when conditioned on $\Gamma$, is a time-homogeneous Markov jump process with random transition intensities given by

\begin{align}
 \Sigma(\ell_1, w_1, i_1; \ell_2, w_2, i_2)=
 \left\{\begin{array}{lll} \gamma\,Q_{i_1, i_2} (\ell_1, w_1) & \mbox{if}& \begin{array}{l}\ell_2 = \ell_1 + 1, w_2=0, i_2\neq i_1,\; \mbox{or}\\ \ell_2 = \ell_1 + 1, w_2=w_1+1, i_2= i_1, \end{array} \\
 0 && \mbox{other } (\ell_1, w_1, i_1)\neq(\ell_2, w_2, i_2),\\
  -\gamma && (\ell_1, w_1, i_1)=(\ell_2, w_2, i_2).\end{array}\right.
 \label{eq:intensities_aux_1}
 \end{align}

where 
\begin{equation}\label{eq:defQ1}\bm{Q}(\ell, w) = \bm{I} + \tfrac{1}{\gamma}\bm{\Lambda}(\chi_{\ell+1}, \chi_{\ell+1}-\chi_{\ell-w})\quad\mbox{for}\quad \ell \ge w.\end{equation}
\end{theorem}
\begin{proof}
We start by computing the law of the process $\{X_{\ell}\}_{\ell\ge 0}$ conditional on $\Gamma$, where $X_{\ell}:=(\beta_\ell, \eta_\ell)$. First, 
\begin{align*}
\P&(X_{\ell+1}=(w_2,i_2)\mid \Gamma, X_0, X_1,\dots, X_{\ell-1}, X_{\ell}=(w_1,i_1))\\
& = \P(X_{\ell+1}=(w_2,i_2)\mid \Gamma, X_{\ell}=(w_1,i_1)),
\end{align*}
where we implicitly used the strong Markov property for the process $\{t,U(t),Z(t)\}_{t\ge 0}$ at time $\chi_\ell$. Consequently, when conditioned on $\Gamma$, $\{X_{\ell}\}_{\ell\ge 0}$ is a (time-inhomogeneous) Markov chain. Now, note that the values for $\Gamma$ and $X_{\ell}$ characterize those of $\chi_{\ell+1}$ and $U(\chi_{\ell+1}-)$, the latter being
\begin{equation}\label{eq:USaux1}
U(\chi_{\ell+1}-)=\chi_{\ell+1} - \chi_{\ell-\beta_\ell}.
\end{equation} 
Using shorthand notation $\{\chi_{\ell+1}\in\Gamma_0\}$ for the event where the arrival $\chi_{\ell+1}$ is associated with $\Gamma_0$ (as opposed to $\Gamma'$), then
\begin{align*}
\P&(X_{\ell+1}=(w_2,i_2)\mid X_{\ell}=(w_1,i_1), \chi_{\ell+1}=y, U(\chi_{\ell+1}-)= s)\\
&  = \P(X_{\ell+1}=(w_2,i_2)\mid X_{\ell}=(w_1,i_1), \chi_{\ell+1}=y, U(\chi_{\ell+1}-)= s, \chi_{\ell+1}\in\Gamma_0)\P(\chi_{\ell+1}\in\Gamma_0)\\
& \quad + \P(X_{\ell+1}=(w_2,i_2)\mid X_{\ell}=(w_1,i_1), \chi_{\ell+1}=y, U(\chi_{\ell+1}-)= s, \chi_{\ell+1}\notin\Gamma_0)\P(\chi_{\ell+1}\notin\Gamma_0)\\
&= \left(\delta_{w_1+1,w_2} \delta_{i_1,i_2} \left(1+ \frac{\Lambda_{i_1 i_2}(y,s)}{\gamma_0}\right) + \delta_{0,w_2}\delta_{i_1,i_2}^c \left(\frac{\Lambda_{i_1 i_2}(y,s)}{\gamma_0}\right)\right)\frac{\gamma_0}{\gamma} + \delta_{w_1+1,w_2}\delta_{i_1,i_2}\frac{\gamma-\gamma_0}{\gamma}\\
&= \delta_{w_1+1,w_2}\delta_{i_1,i_2} + (\delta_{w_1+1,w_2}\delta_{i_1,i_2} + \delta_{0,w_2}\delta_{i_1,i_2}^c)\left(\frac{\Lambda_{i_1 i_2}(y,s)}{\gamma}\right),
\end{align*}
where $\delta_{i_1,i_2}^c = 1 - \delta_{i_1,i_2}$. Thus, employing \eqref{eq:USaux1}, the transition probabilities of $\{X_\ell\}_{\ell\ge 0}$ conditional on $\Gamma$ are given by
\begin{align*}
\P&(X_{\ell+1}=(w_2,i_2)\mid \Gamma, X_{\ell}=(w_1,i_1))\\
&=\delta_{w_1+1,w_2}\delta_{i_1,i_2} + (\delta_{w_1+1,w_2}\delta_{i_1,i_2} + \delta_{0,w_2}\delta_{i_1,i_2}^c)\left(\frac{\Lambda_{i_1 i_2}(\chi_{\ell+1},\chi_{\ell+1}-\chi_{\ell-w_1})}{\gamma}\right).
\end{align*}

To ``homogenize'' the time-inhomogeneous Markov chain $\{X_\ell\}_{\ell\ge 0}$, we consider the augmented process $\{(\ell,\beta_\ell,\eta_\ell)\}_{\ell\ge 0}$, which is easily seen to follow the transition probabilities
\begin{align}
\P&((\ell+1,\beta_{\ell+1},\eta_{\ell+1})=(\ell_2, w_2,i_2)\mid (\ell,\beta_\ell,\eta_\ell)=(\ell_1,w_1,i_1))\nonumber\\
& = \delta_{\ell_1+1,\ell_2}\P(X_{\ell_1+1}=(w_2,i_2)\mid X_{\ell_1}=(w_1,i_1)).\label{eq:transition_aux_1}
\end{align}
Applying the uniformization method (see e.g. \cite{van2018uniformization}) to the Markov chain $\{(\ell,\beta_\ell,\eta_\ell)\}_{\ell\ge 0}$ subordinated at the independent Poisson process $\ap{\Gamma}$ with intensity $\gamma$, yields the Markov jump process $\{(\ap{\Gamma}(s),\beta_{\ap{\Gamma}(s)},\eta_{\ap{\Gamma}(s)})\}_{s\ge 0}= (\ap{\Gamma}, \ap{M}, \ap{Z})$. The jump intensities of this process correspond to the product of the transition probabilities \eqref{eq:transition_aux_1} with the intensity $\gamma$, which finally yields \eqref{eq:intensities_aux_1}.
\end{proof}

Below, we state an important result that formally establishes {$\ap{Z}$}, where $\ap{U}=\{\ap{U}(t)\}_{t\ge 0}$ is the duration process associated to $\ap{Z}$ via 
\begin{equation}\label{eq:Uapp6}\ap{U}(t)=\inf\{r>0: \ap{Z}(t-r)\neq \ap{Z}(t)\},\end{equation} as an appropriate approximation to the semi-Markov process $Z$. Its proof is included in the Appendix \ref{sec:strongSMJP}.

\begin{theorem}\label{th:strong_simp}
$(\ap{Z},\ap{U})$ converges strongly to $(Z,U)$ in Skorokhod's $J_1$-topology as $\gamma\rightarrow\infty$.
\end{theorem}

As discussed in \cite[Lemma 1.2]{prokhorov1956convergence}, the almost sure $J_1$ convergence implies the convergence of the associated measures with respect to the L\'evy-Prokhorov distance. In particular, this guarantees that the probability measure of the stochastic process $(\ap{Z},\ap{U})$ converges to that of $(Z,U)$. This result is significant as it not only ensures the convergence of their transition probabilities in a pointwise manner but also implies that the expectation of any bounded continuous functional evaluated at $\ap{Z}$ will converge to that of $Z$.

\subsection{Efficient computation of the transition probabilities}

In view of Theorem \ref{th:strong_simp}, we now use the process $(\ap{Z},\ap{U})$ to develop an approximating scheme for the transition probabilities of the original process $(Z,U)$. For $0\le t \le s$, $0\le u \le s$, and $0\le v\le t$, define the grid-conditional transition probabilities $\ap{p}(s,\dd v)=\{\ap{p}_{ij}(s,\dd v)\}_{i,j\in \mathcal{J}}$ where
\[\ap{p}_{ij}(s,v) = \P(\ap{Z}(s)=j, \ap{U}(s)\le v \mid \Gamma, \, \ap{Z}(0)=i,\, \ap{U}(0)=0).\]
Due to the doubly-infinite-dimensional structure of the state-space of $(\ap{\Gamma},\ap{M},\ap{Z})$, computing $\ap{p}_{ij}(s,\dd v)$ via raw matrix-exponentiation using Theorem \ref{th:apz-matrix-exp} is, in general, intractable. Below, we present a recursive algorithm that avoids this. In the following, $\mathrm{Poi}_\lambda(k)$ for $k\in\N_0$ takes the form $e^{-\lambda} (\lambda)^k/k!$, and $\mathrm{Erl}_{k,\lambda}(x)$ for $x\ge 0$ takes the form $\lambda(\lambda x)^{k-1}e^{-\lambda x}/(k-1)!$.

\begin{theorem}\label{conditional_transitions_approx}
The measure $\ap{p}_{ij}(s, \dd v)$ can be explicitly computed via
\begin{align}\label{ltp_1}
\ap{p}_{ij}(s, \dd v) = \sum_{\ell, w\,:\, w\le \ell} \ap{p}_{ij}(\ell, w; s, \dd v).
\end{align}
The summands consist of an atomic part at $v=s$ described by 
\begin{align} \ap{p}_{ij}(\ell, w; s, \dd v)
& = \left\{\begin{array}{lll}
\mathrm{Poi}_{\gamma s}(\ell) \times \Pi_{ii}(\ell, w)&\mbox{if}& i=j, \ell=w\ge 0\\
0&&\mbox{otherwise},\end{array}\right.\label{eq:pijatom1}
\end{align}

and an absolutely continuous part described by 
\begin{align} \ap{p}_{ij}(\ell, w; s, \dd v)
& = \left\{\begin{array}{lll}
\mathrm{Erl}_{\ell-w,\gamma}(s-v) \times \mathrm{Poi}_{\gamma v} (w)\times \Pi_{ij}(\ell, w) \dd v &\mbox{if}& \ell>w\ge 0, v< s\\
0&&\mbox{otherwise},\end{array}\right.\label{eq:pijdensity}
\end{align}
where $\bm{\Pi}(\ell,w)=\{\Pi_{ij}(\ell,w)\}_{ij}$ may be computed via the recursion \eqref{eq:rec_aux_1}-\eqref{eq:rec_aux_4}.
\begin{align}
\bm{\Pi}(0,0)&=\bm{I}\label{eq:rec_aux_1}\\
\bm{\Pi}(k,v)&=\bm{0}\quad\mbox{for all}\quad v>k\label{eq:rec_aux_2}\\
\bm{\Pi}(k,0)& = \sum_{k'=0}^{k-1}  \bm{\Pi}(k-1,k')\bm{Q}^{\mathrm{n}}(k-1,k')\quad\mbox{for all}\quad k\ge 1\label{eq:rec_aux_3}\\
\bm{\Pi}(k,v)& = \bm{\Pi}(k-1,v-1)\bm{Q}^{\mathrm{d}}(k-1,v-1)\quad\mbox{for all}\quad k\ge v\ge 1.\label{eq:rec_aux_4}
\end{align}
Here, 
\[\bm{Q}^{\mathrm{d}}(\ell, w)=\begin{pmatrix} Q_{11}(\ell, w)& &&\\&Q_{22}(\ell, w)&&\\ &&\ddots& \\ &&& Q_{JJ}(\ell, w)\end{pmatrix}\quad\mbox{and}\quad \bm{Q}^{\mathrm{n}}(k, w) = \bm{Q}(k, w) - \bm{Q}^{\mathrm{d}}(k, w).\]
\end{theorem}
\begin{proof}
Define the enlarged transition probabilities 
\begin{align*}\ap{p}_{ij}&(\ell, w; s, \dd v) = \P(\ap{\Gamma}(s)=\ell, \ap{M}(s)=w, \ap{Z}(s)=j, \ap{U}(s)\in\dd v \mid \Gamma,\, \ap{Z}(0)=i,\, \ap{U}(0)=0).
\end{align*}
By the law of total probability, it is clear that Equation \eqref{ltp_1} holds. Now, let $X_k=(\beta_k,\eta_k)$ be defined as in the proof of Theorem \ref{th:apz-matrix-exp}, and let
\begin{equation}\label{eq:Pidef1}\Pi_{ij}(\ell,w):=\P(X_{\ell}=(w,j)\mid \Gamma,\,X_{0}=(0,i)),\quad \ell,w\ge 0;\end{equation}
in the following we relate $\Pi_{ij}(\ell,w)$ with $\ap{p}_{ij}(\ell, w; s, \dd v)$, and show that it follows the recursion \eqref{eq:rec_aux_1}-\eqref{eq:rec_aux_4}.

\textbf{Case $v=s$.} Under this event, $\ap{Z}$ does not change states in the interval $[0,s]$, so the only non-zero possibility is $i=j$ and $\ell=w$, which we study next.
Equation \eqref{eq:pijatom1} then follows by conditioning on the number of jumps of $\ap{\Gamma}$ in $[0,s]$, (which follows a probability mass function  $\mathrm{Poi}_{\gamma s}(\ell)$) and multiplying by the grid-conditional probability of 
\[\{X_{\ell'}*\notin\{0\}\times\mathcal{J}\mbox{ and }\; X_{\ell'}*\in\N_0\times\{i\}\mbox{ for all }1\le\ell'\le \ell\;\}=\{X_{\ell}=(w,j)\},\]
where we employed that $\ell=w$ and thus $\eta_{0}=\eta_{1}=\cdots=\eta_\ell=i$.

\textbf{Case $0<v<s$.} In this event, a restart is bound to happen in $s-v\in [0,s]$, so the only non-zero case is when $\ell>w$. Furthermore, $\ap{\Gamma}(s-v-)\neq \ap{\Gamma}(s-v)=\ell-w$, $\ap{M}(s-v)=0$, and $\ap{Z}(s-v-)\neq \ap{Z}(s-v)=j$. After $s-v$, $\ap{Z}$ needs to remain in $j$ for $w$ incremental steps of $\ap{\Gamma}$ (or $\ap{M}$). In other words, by conditioning on the $(\ell-w)$th arrival of $\ap{\Gamma}$ at $s-v$ (which follows an $\mathrm{Erl}_{\ell-w,\gamma}(x)$) and on $w$ arrivals of $\ap{\Gamma}$ occurring in the interval $(s-v, s]$ (occurring with probability $\mathrm{Poi}_{\gamma v} (w)$), and multiplying by the grid-conditional probability of $\{X_{\ell}=(w,i)\}$ (event that captures the aforementioned dynamics), we obtain Equation \eqref{eq:pijdensity}.

We now discuss how the recursion \eqref{eq:rec_aux_1}-\eqref{eq:rec_aux_4} arises. Equation \eqref{eq:rec_aux_1} is a straightforward initial condition, \eqref{eq:rec_aux_2} describes the impossibility of having more arrivals associated with $\ap{M}$ than with $\ap{\Gamma}$. Concerning Equations \eqref{eq:rec_aux_3} and \eqref{eq:rec_aux_4}, they are consequences of conditioning one step prior to $\ell$ (in terms of the process $\{X_k\}_{k\ge 0}$ in \eqref{eq:Pidef1}) and employing the law of total probability.
\end{proof}

\subsection{Value function evaluation}

Now that we have a tractable expression for the grid-conditional transition probabilities of $(\ap{Z},\ap{U})$, we employ these to provide approximations to the value function $V(t; i,u)$. However, without loss of generality, we focus on obtaining approximations to {$V_{i}:=V(0; i,0)$} only: time-shifting arguments akin to those presented in Subsection \ref{sec:SMJP} yield the general case.

Thus, we define the grid-conditional approximation as
 \begin{align}
\ap{V}_i& = \E\left[\int_0^T e^{-{ \int_0^s}r}\dd \ap{B}(s) \mid \Gamma, { \ap{Z}(0)=i,\ap{U}(0)=0}\right], \\
\dd \ap{B}(t)&= \dd B^{\ap{Z}(t),\ap{U}(t)} + \sum_{k \neq \ap{Z}(t-)}b^{(\ap{Z}(t-),\ap{U}(t-)),(k,0)}(t) \dd \ap{N}^{k}(t), \\
\ap{N}^{k}& = \sum_{s\in[0,t]} \1\{\ap{Z}(s)=k,\ap{Z}(s-)\neq k\}.
\end{align}
Based on our findings in Theorem \ref{conditional_transitions_approx}, we now provide a remarkably tractable expression for $\ap{V}_i$.


\begin{corollary}\label{cor:appValue1}
For $i,j,k\in\mathcal{J}$,$j\neq k$ and $s\ge 0$, let
\begin{align}
\ap{p}_{ij}(s, \dd v) & = \ap{p}_{ij}^{\mathrm{a}}(s)\delta_{s}(v) + \ap{p}_{ij}^{\mathrm{c}}(s, v)\dd v,\label{eq:aptpij}\\
\ap{p}_{i;jk}(s, \dd v) & = \ap{p}_{i;jk}^{\mathrm{a}}(s)\delta_{s}(v) + \ap{p}_{i;jk}^{\mathrm{c}}(s, v)\dd v,\label{eq:aptpijk}
\end{align}
where
\begin{align*} 
&\ap{p}_{ij}^{\mathrm{a}}(s) = \sum_{\ell\ge 0} \ap{p}_{ij}^{\mathrm{a}}(\ell; s),\qquad\ap{p}_{i;jk}^{\mathrm{a}}(s) = \sum_{\ell\ge 0} \ap{p}_{ij}^{\mathrm{a}}(\ell; s)(\gamma Q_{jk}(\ell,\ell)),\\
&\ap{p}_{ij}^{\mathrm{c}}(s,v) = \sum_{\ell, w\,:\, w< \ell} \ap{p}_{ij}^{\mathrm{c}}(\ell, w; s,v),\qquad \ap{p}_{i;jk}^{\mathrm{c}}(s,v) = \sum_{\ell, w\,:\, w< \ell} \ap{p}_{ij}^{\mathrm{c}}(\ell, w; s,v)(\gamma Q_{jk}(\ell,w)),\\
&\ap{p}_{ij}^{\mathrm{a}}(\ell; s) = \mathrm{Poi}_{\gamma s}(\ell)\; {\Pi}_{ii}(\ell, \ell)\;\mathds{1}_{\{i=j\}},\\
&\ap{p}_{ij}^{\mathrm{c}}(\ell, w; s,v) = \mathrm{Erl}_{\ell-w,\gamma}(s-v) \; \mathrm{Poi}_{\gamma v} (w)\; {\Pi}_{ij}(\ell, w)\;\mathds{1}_{\{\ell>w\ge 0, v<s\}}.
\end{align*}
Then,
\begin{align}
\ap{V}_{i}& = \int_{s=0}^T e^{-\int_0^s r} \sum_{j\in\mathcal{J}}\int_{v=0}^s \ap{p}_{ij}(s,\dd v) \dd B^{j,v}(s)\nonumber\\
&\qquad + \int_{s=0}^T e^{-\int_0^s r} \sum_{\substack{j,k\in\mathcal{J}\\ j\neq k}}\int_{v=0}^s \ap{p}_{i;jk}(s,\dd v) b^{(j,v),(k,0)}(s)\dd s.\label{eq:Vi1}
\end{align}
\end{corollary}
\begin{proof}
First, note that \eqref{eq:aptpij} is simply a re-statement of Theorem \ref{conditional_transitions_approx}.
The first summand in the r.h.s. of \eqref{eq:Vi1} follows by noting that
\begin{align*}
& \E\Big(\int_0^T e^{-{ \int_0^s} r}\sum_{j\in\mathcal{J}} \mathds{1}_{\{\ap{Z}(t)=j\}} \dd B^{j,\ap{U}(s)} (s) \;\Big|\; \Gamma, \ap{Z}(0)=i, \ap{U}(0)=0 \Big)\\
& \quad = \int_0^T e^{-{ \int_0^s} r}\sum_{j\in\mathcal{J}} \E\left(\left. \mathds{1}_{\{\ap{Z}(t)=j\}} \dd B^{j,\ap{U}(s)} (s) \;\right|\; \Gamma, \ap{Z}(0)=i, \ap{U}(0)=0 \right)\\
& \quad = \int^T_{s=0} e^{-\int_0^s r} \sum_{j\in\mathcal{J}}\int_{v=0}^s \ap{p}_{ij}(s,\dd v) \dd B^{j,v}(s) .
\end{align*}

For the second summand, observe that for $j\neq k$,
\[\ap{N}_{jk}(t)= \sum_{\substack{\ell,w\\ w\le \ell}} \ap{N}_{jk;\ell,w}(t),\]
where $\ap{N}_{jk;\ell,w}$ represents the counting process linked to the jumps of $\ap{Z}$ from $j$ to $k$ occurring while $\ap{\Gamma}$ is $\ell$ and $\ap{M}$ is $w$ (before the jump). This counting process, $\ap{N}_{jk;\ell,w}$, has a jump intensity described by $\mathds{1}_{\ap{N}(s-)=\ell, \ap{M}(s-)= w, \ap{Z}(s-)=j} (\gamma Q_{jk}(\ell,w))$. By the general theory of counting processes (see \cite{last1995marked} and more generally \cite{bremaud1981point}) this leads to

\begin{align*}
\E&\Big(\int_0^T e^{-{ \int_0^s} r} \sum_{\substack{j,k\in\mathcal{J}\\j\neq k}}b^{(j,\ap{U}(s-)),(k,0)}(s)\,\dd \ap{N}_{jk}(s) \;\Big|\; \Gamma, \ap{Z}(0)=i, \ap{U}(0)=0 \Big)\\
& \quad = \int_0^T e^{-{ \int_0^s} r} \sum_{\substack{j,k\in\mathcal{J}\\j\neq k}} \sum_{\substack{\ell,w\\ w\le \ell}} \E\left(\left. b^{(j,\ap{U}(s-)),(k,0)} (s)\,\dd \ap{N}_{jk;\ell, w}(s) \;\right|\;\Gamma, \ap{Z}(0)=i, \ap{U}(0)=0 \right)\\
& \quad = \int_{s=0}^T e^{-\int_0^s r} \sum_{\substack{j,k\in\mathcal{J}\\ j\neq k}}\sum_{\substack{\ell,w\\w\le \ell}} \int_{v=0}^s \ap{p}_{ij}(\ell,w;s,\dd v) (\gamma Q^{\mathrm{n}}_{jk}(\ell, w))b^{(j,v),(k,0)}(s)\,\dd s\\
&\quad = \int_{s=0}^T e^{-\int_0^s r} \sum_{\substack{j,k\in\mathcal{J}\\, j\neq k}}\int_{v=0}^s \ap{p}_{i;jk}(s,\dd v) b^{(j,v),(k,0)}(s)\dd s,
\end{align*}

where the final equality is a result of reorganizing the terms in \eqref{eq:aptpijk}.

\end{proof}

{
A similar approximation holds for the expected cashflow { $\mathcal{C}_i:=\mathcal{C}_{0;i,0}$} in \eqref{cashflow_definition} by defining
\begin{align*}
\ap{C}_i(\dd s) = \E\left[\dd \ap{B}(s) \mid \Gamma, { \ap{Z}(0) = i, \ap{U}(0) = 0}\right].
\end{align*}
An analogous procedure to the one that leads to \eqref{eq:Vi1} yields the formula
\begin{align}\label{cashflow_approximation}
\ap{C}_{i}(\dd s) &= \sum_{j \in \mathcal{J}} \int_{v=0}^s \ap{p}_{ij}(s, \dd v) \, \dd B^{j,v}(s) + \left( \sum_{\substack{j,k \in \mathcal{J} \\ j \neq k}} \int_{v=0}^s \ap{p}_{i;jk}(s, \dd v) \, b^{(j,v),(k,0)}(s) \right) \dd s.
\end{align}
}

To conclude this section, we emphasize that most of the results presented up to this point concern \emph{grid-conditional} probabilities or expectations. While these hold under fairly robust assumptions, dealing with conditional quantities implies that the formulas are, in essence, random variables that converge to a constant when $\gamma\rightarrow \infty$. Although applying a Monte Carlo method would accelerate the convergence to the desired descriptor, this would inevitably require more computing power. In the next section, under certain regularity conditions, we provide an alternative approximation that circumvents the need for grid-conditional quantities and offers unconditional ones instead.

\section{Unconditional approximations of semi-Markov processes}\label{sec:unconditional} 

Note that the grid-conditional transition probabilities $\ap{p}_{ij}$ in Theorem \ref{conditional_transitions_approx} inherit their randomness from the matrices $\bm{Q}^{\mathrm{d}}(\ell,w)$ and $\bm{Q}^{\mathrm{n}}(\ell,w)$, themselves being defined through the matrix $\bm{Q}(\ell,w)=\bm{I} + \tfrac{1}{\gamma}\bm{\Lambda}(\chi_{\ell+1}, \chi_{\ell+1}-\chi_{\ell-w})$. Here we propose to use the matrix $\tld{\bm{Q}}(\ell,w)$ instead of $\bm{Q}(\ell,w)$, where
\begin{align}\tld{\bm{Q}}(\ell,w) = \bm{I} + \tfrac{1}{\gamma}\bm{\Lambda}\left(\tfrac{\ell+1}{\gamma}, \tfrac{w+1}{\gamma}\right).
\end{align}
Note that we are effectivelly removing the randomess by replacing $\chi_{\ell+1}$ and $\chi_{\ell+1}-\chi_{\ell-w}$ with their expected values, $\tfrac{\ell+1}{\gamma}$ and $\frac{w+1}{\gamma}$. Heuristically, the strong law of large numbers implies that, indeed, each $\chi_\ell$ converges to their expected value. In fact, as evidenced by \cite[Lemma 1, Eq (17)]{bladt_peralta} -- included here in the formula \eqref{eq:grid-conv1} of Appendix \ref{sec:convergence_densities} -- the stochastic grid $\Gamma=\{\chi_\ell\}_{\ell\ge 0}$ converges uniformly to the deterministic grid $\{\ell/\gamma\}_{\ell\ge 0}$ over increasing compact intervals. As expected, in order for $\tld{\bm{Q}}(\ell, w)$ to inherit this convergence, we need some type of continuity assumption for $\bm{\Lambda}$. In this section we assume that $\bm{\Lambda}(s,v)=\{\Lambda_{ij}(s,v)\}$ is entrywise Lipschitz continuous, { equivalent to assuming that} there exists some $K>0$ such that
{
\begin{align}
\sup_{i\in\mathcal{J}}\sum_{j\in\mathcal{J}}\left| \Lambda_{ij}(s_1,v_1) - \Lambda_{ij}(s_2,v_2)\right|\le K\left(|s_1-s_2| + |v_1-v_2|\right).\label{eq:LipschitzLambda1}
\end{align}
}
Under the assumption \eqref{eq:LipschitzLambda1},
{
\begin{align}
\sup_{i\in\mathcal{J}}\sum_{j\in\mathcal{J}} \left|Q_{ij}(\ell, w)-\tld{Q}_{ij}(\ell, w)\right|& =\gamma^{-1}\sup_{i\in\mathcal{J}}\sum_{j\in\mathcal{J}}\left|\Lambda_{ij}(\chi_{\ell+1}, \chi_{\ell+1}-\chi_{\ell-w})-\Lambda_{ij}\left(\tfrac{\ell+1}{\gamma}, \tfrac{w+1}{\gamma}\right)\right|\nonumber\\
& \le \gamma^{-1} K \left[\left|\chi_{\ell+1}-\tfrac{\ell+1}{\gamma}\right| + \left|\left(\chi_{\ell+1}-\chi_{\ell-w}\right) - \left(\tfrac{w+1}{\gamma}\right)\right|\right]\nonumber\\
& \le \gamma^{-1} K \left[\left|\chi_{\ell+1}-\tfrac{\ell+1}{\gamma}\right| + \left|\left(\chi_{\ell+1}-\chi_{\ell-w}\right) - \left(\tfrac{\ell+1}{\gamma} - \tfrac{\ell-w}{\gamma}\right)\right|\right]\nonumber\\
& \le \gamma^{-1} K \left[\left|\chi_{\ell+1}-\tfrac{\ell+1}{\gamma}\right| + \left|\chi_{\ell+1}-\tfrac{\ell+1}{\gamma}\right| + \left|\chi_{\ell-w} -  \tfrac{\ell-w}{\gamma}\right|\right]\nonumber\\
&\le \gamma^{-1}3K\sup_{\ell'\le \ell+1}\left|\chi_{\ell}-\tfrac{\ell}{\gamma}\right|,\label{eq:Qbound7}
\end{align}
}
confirming that $\tld{\bm{Q}}(\ell,w)$ converges to $\bm{Q}(\ell,w)$, at an $o(\gamma^{-1})$ rate, as long as the maximum distance between $\chi_{\ell'}$ and $\ell'/\gamma$ for all $\ell'\le \ell +1$ converges to $0$, which is guaranteed by \eqref{eq:grid-conv1} of Appendix \ref{sec:convergence_densities}. 

The next step is identifying how the convergence of $\tld{\bm{Q}}(\ell,w)$ to $\bm{Q}(\ell,w)$ implies the convergence of the associated transition probabilities. For this, consider now a Markov chain $\{\tld{X}_k\}_{k\ge 0}$ over the state space $\N_0\times \mathcal{J}$ driven by the transition probabilities
\begin{align*}\P \left(\tld{X}_{k+1}=(w',j)|\tld{X}_{k}=(w,i)\right) = \left\{\begin{array}{ccc}\tld{Q}^{\mathrm{d}}_{ii}(k,w)&\mbox{ if }&i=j, w'=w+1\ge 1\\
\tld{Q}^{\mathrm{n}}_{ij}(k,w)&\mbox{ if }&i\neq j, w'= 0, w\ge 0\\
0&&\mbox{otherwise},\end{array}\right.\end{align*}
where 
\[\tld{\bm{Q}}^{\mathrm{d}}(\ell, w)=\begin{pmatrix} \tld{Q}_{11}(\ell, w)& &&\\&\tld{Q}_{22}(\ell, w)&&\\ &&\ddots& \\ &&& \tld{Q}_{JJ}(\ell, w)\end{pmatrix}\quad\mbox{and}\quad \tld{\bm{Q}}^{\mathrm{n}}(k, w) = \tld{\bm{Q}}(k, w) - \tld{\bm{Q}}^{\mathrm{d}}(k, w).\]
Now, we uniformize the three-dimensional process $\{(k,X_k)\}_{k\ge 0}$ over a Poisson process $\apt{\Gamma}$ of intensity $\gamma$, resulting in a continuous-time Markov jump process $(\apt{\Gamma},\apt{M},\apt{Z})$. By arguments analogous to those presented in Theorem \ref{conditional_transitions_approx}, we get that the associated transition probabilities { are explicitly computable}. Specifically, 
\[\apt{p}_{ij}(s, v) = \P(\apt{Z}(s)=j, \apt{U}(s)\le v \mid \, \apt{Z}(0)=i,\, \apt{U}(0)=0)\]
with 
\begin{equation*}\apt{U}(t)=\inf\{r>0: \apt{Z}(t-r)\neq \apt{Z}(t)\},\end{equation*}
can be computed via 
\begin{align}
\apt{p}_{ij}(s, \dd v) & = \apt{p}_{ij}^{\mathrm{a}}(s)\delta_{s}(v) + \apt{p}_{ij}^{\mathrm{c}}(s, v)\dd v,\label{ltp_2}
\end{align}
where
\begin{align*} 
&\apt{p}_{ij}^{\mathrm{a}}(s) = \sum_{\ell\ge 0} \ap{p}_{ij}^{\mathrm{a}}(\ell; s),\quad\apt{p}_{ij}^{\mathrm{c}}(s,v) = \sum_{\ell, w\,:\, w< \ell} \apt{p}_{ij}^{\mathrm{c}}(\ell, w; s,v),\\
&\apt{p}_{ij}^{\mathrm{a}}(\ell; s) = \mathrm{Poi}_{\gamma s}(\ell)\; \tld{\Pi}_{ii}(\ell, \ell)\;\mathds{1}_{\{i=j\}},\\
&\apt{p}_{ij}^{\mathrm{c}}(\ell, w; s,v) = \mathrm{Erl}_{\ell-w,\gamma}(s-v) \; \mathrm{Poi}_{\gamma v} (w)\; \tld{\Pi}_{ij}(\ell, w)\;\mathds{1}_{\{\ell>w\ge 0, v<s\}}.
\end{align*}
Here $\tld{\bm{\Pi}}(\ell,w)=\{\tld{\Pi}_{ij}(\ell,w)\}_{ij}$ where
\begin{align*}\tld{\Pi}_{ij}(\ell,w):=\P(\tld{X}_{\ell}=(w,j)\mid \tld{X}_{0}=(0,i)),\quad \ell,w\ge 0,\end{align*}
follows the recursion
\begin{align}
\tld{\bm{\Pi}}(0,0)&=\bm{I}\label{eq:rec_aux_1t}\\
\tld{\bm{\Pi}}(k,v)&=\bm{0}\quad\mbox{for all}\quad v>k\label{eq:rec_aux_2t}\\
\tld{\bm{\Pi}}(k,0)& = \sum_{k'=0}^{k-1}  \tld{\bm{\Pi}}(k-1,k')\tld{\bm{Q}}^{\mathrm{n}}(k-1,k')\quad\mbox{for all}\quad k\ge 1\label{eq:rec_aux_3t}\\
\tld{\bm{\Pi}}(k,v)& = \tld{\bm{\Pi}}(k-1,v-1)\tld{\bm{Q}}^{\mathrm{d}}(k-1,v-1)\quad\mbox{for all}\quad k\ge v\ge 1.\label{eq:rec_aux_4t}
\end{align}

Then we have the following convergence result regarding the densities associated to $\ap{p}_{ij}(\ell, w; s, \dd v)$ and $\apt{p}_{ij}(\ell, w; s, \dd v)$; we include its proof in Appendix \ref{sec:convergence_densities}.

\begin{theorem}\label{th:convergence_densities}
For any $q>1$, $\varepsilon > 0$, there exists a constant $\alpha(q,\varepsilon)$ such that { for all $s\ge v\ge 0$},
\begin{align*}
&\P\left({ \sup_{i\in\mathcal{J}}\sum_{j\in\mathcal{J}}}\left|\ap{p}_{ij}^{\mathrm{a}}(s) - \apt{p}_{ij}^{\mathrm{a}}( s)\right|\le (s+1)\,\alpha(q,\varepsilon)\,(\log \gamma)\gamma^{-1/2 + \varepsilon/2}\right)=o(\gamma^{-q}),\\
&\P\left({ \sup_{i\in\mathcal{J}} \int_0^v \sum_{j\in\mathcal{J}}\left| \ap{p}_{ij}^{\mathrm{c}}(s,v') - \apt{p}_{ij}^{\mathrm{c}}(s,v')\right| \dd v'}\le (s+1)\,\alpha(q,\varepsilon)\,(\log \gamma)\gamma^{-1/2 + \varepsilon/2}\right)=o(\gamma^{-q}).
\end{align*}
\end{theorem}

{ Particularly, Theorem \ref{th:convergence_densities} implies that the distance between the transition probabilities of $\ap{Z}$ and $\apt{Z}$ converges almost surely to $0$. Since Theorem \ref{th:strong_simp} establishes that the transition probabilities of $\ap{Z}$ converge to those of $Z$, the same holds for $\apt{Z}$. This result {is similar to} Conjecture 4.7 of \cite{ahmad2023aggregate}, where the authors define a class of jump processes called aggregated Markov models with resets, positing that these models are dense within the set of semi-Markov models. {We verify a similar result, namely that the class of (infinite-state-space) time-homogeneous Markov jump processes is sufficiently robust to approximate semi-Markov processes arbitrarily well. We next provide the statement of the precise result.}

{
\begin{corollary}
Fix an intensity-driven semi-Markov process \( Z \) on a state-space \(\mathcal{J}\) with bounded jump intensity \(\bm{\Lambda}(\cdot, \cdot)\). There exists a sequence of grid-conditional, infinite-block, time-homogeneous Markov jump processes \(\{(\ap{\Gamma}^\gamma, \ap{M}^\gamma, \ap{Z}^\gamma)\}_{\gamma > 0}\) defined over the state-space \(\mathbb{N}_0 \times \mathbb{N}_0 \times \mathcal{J}\), with intensities given by \eqref{eq:intensities_aux_1}, such that \(\ap{Z}^\gamma\) converges weakly (in the \( J_1 \)-topology sense) to \( Z \) as \(\gamma \to \infty\). Furthermore, under the Lipschitz assumption \eqref{eq:LipschitzLambda1}, the unconditional infinite-block, time-homogeneous Markov jump processes \(\{(\ap{\Gamma}^\gamma, \ap{M}^\gamma, \ap{Z}^\gamma)\}_{\gamma > 0}\) satisfy that the transition probabilities of \(\ap{Z}^\gamma\) converge pointwise to the transition probabilities of \( Z \) as \(\gamma \to \infty\).
\end{corollary}
}

{ In comparison to the methodology of \cite{ahmad2023aggregate}, our approach is closely related to their aggregate Markov models. Specifically, they examine a class of finite state-space, block-structured jump processes where the duration is reset upon exiting each block. While our model also employs a block structure, it differs by embedding all dependencies in the discrete augmentation of the state-space. We emphasize that, technically, our countably-infinite state-space model does not align with aggregate Markov models. However, with appropriate state-space truncation, our model is indeed a special case of theirs. As such, stochastic payments results in \cite{ahmad2023aggregate} for aggregate Markov models, appearing in connection with policyholder behaviour (e.g. stochastic retirement), carry over directly to our state-space truncated models. Extensions to the non-truncated model is still an open question.}

Finally, having ensured that the unconditional transition probabilities converge to the same 
(almost sure) limit of the grid-conditional ones, and that this limit corresponds to the transition probabilities of the original process $(Z,U)$, we propose using the unconditional transition probabilities to approximate the value of $V_i$ in a manner akin to that of $\ap{V}_i$. Specifically, we consider the approximation $\apt{V}_i$ where
\begin{align*}
\apt{V}_{i}& = \int_{s=0}^T e^{-\int_0^s r} \sum_{j\in\mathcal{J}}\int_{v=0}^s \apt{p}_{ij}(s,\dd v) \dd B^{j,v}(s)\\
&\qquad + \int_{s=0}^T e^{-\int_0^s r} \sum_{\substack{j,k\in\mathcal{J}\\ j\neq k}}\int_{v=0}^s \apt{p}_{i;jk}(s,\dd v) b^{(j,v),(k,0)}(s)\dd s,
\end{align*}
for $\apt{p}_{ij}(s, \dd v)$ defined as in \eqref{ltp_2}, and
\begin{align*}
\apt{p}_{i;jk}(s, \dd v) & = \apt{p}_{i;jk}^{\mathrm{a}}(s)\delta_{s}(v) + \apt{p}_{i;jk}^{\mathrm{c}}(s, v)\dd v,\\
\apt{p}_{i;jk}^{\mathrm{a}}(s) & = \sum_{\ell\ge 0} \apt{p}_{ij}^{\mathrm{a}}(\ell; s)(\gamma \tld{Q}_{jk}(\ell,\ell)),\\
\apt{p}_{i;jk}^{\mathrm{c}}(s,v) &= \sum_{\ell, w\,:\, w< \ell} \apt{p}_{ij}^{\mathrm{c}}(\ell, w; s,v)(\gamma \tld{Q}_{jk}(\ell,w)).
\end{align*}
Note that the approximations $\ap{V}_i$ and $\apt{V}_i$ are asymptotically equivalent and for small precision, virtually indistinguishable.
{
\section{Semi-markov processes on general state-spaces and their approximations}
 Up to this point, we have only dealt with the case the semi-Markov process takes values over a finite state-space, namely, $\mathcal{J}$: this was done (mostly) for pedagogical purposes. In this section, we relax such condition considerably. In fact, employing the appropriate mathematical definitions, next we show how our methods can be generalized to the case $\mathcal{J}$ is a general (measurable) state-space. { Not only is this a theoretical improvement on its own, but it also opens up the possibility of studying practical insurance models requiring a general state-space, see for instance \cite{sokol2015generic} for further constructions in this direction.}

Throughout this section, let $(Z,U)$ be the semi-Markov process taking values on $\mathcal{J}$, a general space endowed with a $\sigma$-algebra $\mathscr{J}$ that contains $\{x\}$ for all $x\in\mathcal{J}$. Here we consider the case in which $Z$ is driven by a family of time- and duration-inhomogeneous intensity (signed) kernel $\Lambda(s,v):$, $s,v\ge 0$, in the sense that
\begin{equation}\label{eq:general1}\mathds{P}(Z(s+h)\in A\mid Z(s)=x,U(s)=v)=\delta_x(A)+\Lambda(s,v)(x,A) h+ o(h), \quad h>0,
\end{equation}
where $\delta_x$ denotes the Dirac measure on $x$. To avoid encountering technical difficulties, here we assume that for all $x\in \mathcal{J}$ and $A\in\mathscr{J}$, the mapping $s\mapsto \Lambda(s,v)(x,A)$ is c\`adl\`ag for all $v\ge 0$, the mapping $v\mapsto \Lambda(s,v)(x,A)$ is c\`adl\`ag for all $s\ge 0$, and that
\begin{align*}
\sup_{s,v\ge 0, x\in\mathcal{J}}|\Lambda|(s,v)(x,\{x\})\le\gamma_0<\infty,
\end{align*}
where $|\Lambda|(s,v)(x,\cdot)$ denotes the total variation measure associated to $\Lambda(s,v)(x,\cdot)$. Note that for the r.h.s. of (\ref{eq:general1}) to correspond to a probability, it must be the case that the Hahn-Jordan decomposition $\Lambda(s,v)(x,\cdot)=\Lambda^+(s,v)(x,\cdot)-\Lambda^-(s,v)(x,\cdot)$, with $\Lambda^+(s,v)(x,\cdot)
$ and $\Lambda^+(s,v)(x,\cdot)$ being nonnegative measures, is such that $\Lambda^+(s,v)(x,\cdot)$ has support on $\mathcal{J}\setminus\{x\}$, and $\Lambda^-(s,v)(x,\cdot)$ has support on $\{x\}$. Particularly,
\[|\Lambda|(s,v)(x,\{x\})=\Lambda^-(s,v)(x,\{x\}).\]
and furthermore, by the law of total probability applied to (\ref{eq:general1}),
\[\Lambda^-(s,v)(x,\{x\})=\Lambda^+(s,v)(x,\mathcal{J}\setminus\{x\}).\]
Now, the construction of $(Z,U)$ with the distributional characteristics (\ref{eq:general1}) can be achieved by the same uniformization method outlined in Section \ref{sec:SMJP}. Namely, with a Poisson process $\Gamma_0$ of intensity $\lambda_0$, arrivals of $\Gamma_0$ denoted by $T_0=0,T_1,T_2,\dots$, and a general state-space Markov chain $\zeta_0, \zeta_1,\zeta_2,\dots$ with transition probabilities given by the respective probability kernels $P_0,P_1,P_2,\dots$, where $P_\ell:\mathcal{J}\times\mathscr{J}\mapsto[0,1]$ is defined by
\[P_\ell(x,A)=\delta_x(A)+\frac{1}{\gamma_0}\Lambda(T_{\ell+1}, \upsilon_{\ell} + (T_{\ell+1}-T_\ell))(x,A),\quad\mbox{where}\]
\begin{align*}
\upsilon_{\ell+1} = \left\{\begin{array}{ccc} \upsilon_{\ell} + (T_{\ell+1}-T_\ell) & \text{if} & \zeta_{\ell+1}=\zeta_{\ell},\\
0 & \text{if} & \zeta_{\ell+1}\neq\zeta_{\ell}. \end{array}\right.
\end{align*}
In the vein of Section \ref{sec:approxgen1}, via an uniformization scheme with increasing Poissonian intensities and subordination with an independent Poisson process, we can construct two approximations: A grid-conditional one resulting in the process $(\ap{Z},\ap{U})$ that strongly converges to $(Z,U)$, and an unconditional one with transition probabilities that converge to those of $(Z,U)$.

\begin{remark}
Specifically, in the grid-conditional case, the process $(\ap{Z},\ap{U})$ converges in the $J_1$-topology, where the topology endowed to the first coordinate is the discrete one, that is, the \emph{largest topology} possible over $\mathcal{J}$. The reason behind this is because our proposed approximation visits the exact same states of the original process, so the distance between the two processes (barring time-changes) in the first coordinate is zero in the $J_1$-metric. Crucially, we note that the second coordinate, the time component, does have a natural Borel structure, under which we show convergence.
\end{remark}
With this setup, the proof in Appendix \ref{sec:strongSMJP} follows verbatim for the case of general state-space. Here, $(\ap{Z},\ap{U})$ has transition probabilities $\ap{p}:\mathcal{J}\times\mathds{R}_+\times\mathscr{J}\times\mathds{B}(\mathds{R}_+)\mapsto [0,1]$,
\[\ap{p}(s,\dd v; x,\dd y)=\mathds{P}(\ap{Z}(s)\in\dd y, \ap{U}(s)\in\dd v\mid \ap{Z}(0)=x, \ap{U}(0)=0)\]
given by
\begin{align*}
\ap{p}(s,\dd v; x,\dd y) & = \ap{p}^{\mathrm{a}}(s; x,\dd y)\delta_{s}(\dd v) + \ap{p}^{\mathrm{c}}(s, v; x,\dd y)\dd v,
\end{align*}
where
\begin{align*} 
&\ap{p}^{\mathrm{a}}(s; x,\dd y) = \sum_{\ell\ge 0} \ap{p}^{\mathrm{a}}(\ell; s; x,\dd y),\quad\ap{p}^{\mathrm{c}}(s,v; x,\dd y) = \sum_{\ell, w\,:\, w< \ell} \ap{p}^{\mathrm{c}}(\ell,w;s,v; x,\dd y),\\
&\ap{p}^{\mathrm{a}}(\ell; s;x,\dd y)= \mathrm{Poi}_{\gamma s}(\ell)\; {\Pi}(\ell, \ell;x,\dd y),\\
&\ap{p}^{\mathrm{c}}(\ell,w;s,v; x,\dd y) = \mathrm{Erl}_{\ell-w,\gamma}(s-v) \; \mathrm{Poi}_{\gamma v} (w)\; {\Pi}(\ell, w)(x,\dd y)\;\mathds{1}_{\{\ell>w\ge 0, v<s\}},
\end{align*}
and the kernel ${\Pi}(\ell, w)(\cdot,\cdot):\mathcal{J}\times\mathscr{J}\mapsto[0,1]$ can be computed recursively via
\begin{align}
{{\Pi}}(0,0)(x,\dd y)&=\delta_x(\dd y)\nonumber\\
{{\Pi}}(k,v)(x,\dd y)&={0}\quad\mbox{for all}\quad v>k\nonumber\\
{\Pi}(k,0)(x,\dd y)& = \sum_{k'=0}^{k-1}  \left({\Pi}(k-1,k'){Q}^{\mathrm{n}}(k-1,k')\right)(x,\dd y)\quad\mbox{for all}\quad k\ge 1\label{eq:gen_rec_aux_1}\\
{\Pi}(k,v)(x,\dd y)& = \left({\Pi}(k-1,v-1){Q}^{\mathrm{d}}(k-1,v-1)\right)(x,\dd y)\quad\mbox{for all}\quad k\ge v\ge 1.\label{eq:gen_rec_aux_2}
\end{align}
with
\begin{align*}
{Q}^{\mathrm{n}}(\ell,w)(x,\dd y) &= \delta_x(\dd y)-\frac{1}{\gamma}\Lambda^-(\chi_{\ell+1}, \chi_{\ell+1}-\chi_{\ell-w})(x,\dd y),\\
{Q}^{\mathrm{d}}(\ell,w)(x,\dd y) &=\frac{1}{\gamma}\Lambda^+(\chi_{\ell+1}, \chi_{\ell+1}-\chi_{\ell-w})(x,\dd y).
\end{align*}
Note that in (\ref{eq:gen_rec_aux_1}) and (\ref{eq:gen_rec_aux_2}) we employed the usual composition of kernels, namely, 
\begin{align*}
\left({\Pi}(k-1,k'){Q}^{\mathrm{n}}(k-1,k')\right)(x,\dd y):& =\int {\Pi}(k-1,k';x,\dd x'){Q}^{\mathrm{n}}(k-1,k';x',\dd y),\\
\left({\Pi}(k-1,v-1){Q}^{\mathrm{d}}(k-1,v-1)\right)(x,\dd y) & := \int {\Pi}(k-1,v-1;x,\dd x'){Q}^{\mathrm{r}}(k-1,v-1;x',\dd y).
\end{align*}

Likewise, we can consider unconditional transition probability approximations. These take the form:
\begin{align*}
\apt{p}(s,\dd v; x,\dd y) & = \apt{p}^{\mathrm{a}}(s; x,\dd y)\delta_{s}(\dd v) + \apt{p}^{\mathrm{c}}(s, v; x,\dd y)\dd v,
\end{align*}
where
\begin{align*} 
&\apt{p}^{\mathrm{a}}(s; x,\dd y) = \sum_{\ell\ge 0} \ap{p}^{\mathrm{a}}(\ell; s; x,\dd y),\quad\apt{p}^{\mathrm{c}}(s,v; x,\dd y) = \sum_{\ell, w\,:\, w< \ell} \apt{p}^{\mathrm{c}}(\ell,w;s,v; x,\dd y),\\
&\apt{p}^{\mathrm{a}}(\ell; s;x,\dd y)= \mathrm{Poi}_{\gamma s}(\ell)\; \tld{\Pi}(\ell, \ell;x,\dd y),\\
&\apt{p}^{\mathrm{c}}(\ell,w;s,v; x,\dd y) = \mathrm{Erl}_{\ell-w,\gamma}(s-v) \; \mathrm{Poi}_{\gamma v} (w)\; \tld{\Pi}(\ell, w)(x,\dd y)\;\mathds{1}_{\{\ell>w\ge 0, v<s\}},
\end{align*}
with
\begin{align*}
\tld{\Pi}(0,0)(x,\dd y)&=\delta_x(\dd y)\\
\tld{\Pi}(k,v)(x,\dd y)&={0}\quad\mbox{for all}\quad v>k\\
\tld{\Pi}(k,0)(x,\dd y)& = \sum_{k'=0}^{k-1}  \left(\tld{\Pi}(k-1,k'){Q}^{\mathrm{n}}(k-1,k')\right)(x,\dd y)\quad\mbox{for all}\quad k\ge 1,\\
\tld{\Pi}(k,v)(x,\dd y)& = \left(\tld{\Pi}(k-1,v-1)\tld{Q}^{\mathrm{d}}(k-1,v-1)\right)(x,\dd y)\quad\mbox{for all}\quad k\ge v\ge 1,\\
\tld{Q}^{\mathrm{n}}(\ell,w)(x,\dd y) &= \delta_x(\dd y)-\frac{1}{\gamma}\Lambda^-\left(\tfrac{\ell+1}{\gamma}, \tfrac{w+1}{\gamma})(x,\dd y\right),\\
\tld{Q}^{\mathrm{d}}(\ell,w)(x,\dd y) &=\frac{1}{\gamma}\Lambda^+\left(\tfrac{\ell+1}{\gamma}, \tfrac{w+1}{\gamma})(x,\dd y\right).
\end{align*}

In Theorem \ref{th:convergence_densities_general} below we establish the convergence of the transition probabilities $\ap{p}$ and $\apt{p}$, under the assumption that there exists $K>0$ such that for all $x\in\mathcal{J}$, $s_1,s_2,v_1,v_2\ge 0$, 
\begin{align}\left|\Lambda(s_1,v_1)-\Lambda(s_2,v_2)\right|(x,\mathcal{J})\le K \left(|s_1-s_2| + |v_1-v_2|\right).
\end{align}
Here, the l.h.s. corresponds to the total variation measure of $\Lambda(s_1,v_1)(x,\cdot)-\Lambda(s_2,v_2)(x,\cdot)$ evaluated at $\mathcal{J}$. Its proof is included in Appendix \ref{sec:GeneralProofs}.

\begin{theorem}\label{th:convergence_densities_general}
For any $q>1$, $\varepsilon > 0$, there exists a constant $\alpha'(q,\varepsilon)$ such that for all $s\ge v\ge 0$,
\begin{align*}
&\P\left(\sup_{x\in\mathcal{J}}\left|\ap{p}^{\mathrm{a}}(s) - \apt{p}^{\mathrm{a}}(s)\right|(x,\mathcal{J})\le (s+1)\,\alpha'(q,\varepsilon)\,(\log \gamma)\gamma^{-1/2 + \varepsilon/2}\right)=o(\gamma^{-q}),\\
&\P\left(\sup_{x\in\mathcal{J}}\int_0^v \left|\ap{p}^{\mathrm{c}}(s,v') - \apt{p}^{\mathrm{c}}(s,v')\right| (x,\mathcal{J}) \dd v'\le (s+1)\,\alpha'(q,\varepsilon)\,(\log \gamma)\gamma^{-1/2 + \varepsilon/2}\right)=o(\gamma^{-q}).
\end{align*}
\end{theorem}
}

\section{Illustration and extensions}\label{sec:examples}
{ This section illustrates the numerical feasibility of our methods. First, we conduct the evaluation of cashflows (the main component of the value function) for a semi-Markov parametrization calibrated with real-life data. The focus here is not on outperforming state-of-the-art numerical integro-differential solvers in terms of precision or speed, but rather on demonstrating the ease of implementation of our approach. Finally, we provide a perspective on potential applications of our stochastic approximation framework in credit risk and controlled semi-Markov models.}

\subsection{Disability model}

Consider an actuarial application by analyzing the disability model with a semi-Markov structure, which is one of the most general stochastic structures implemented in the life-insurance market. This model considers three states: active, disabled, and dead, where all transitions are allowed except that, evidently, death is an absorbing state. The model is conceptually depicted in Figure \ref{fig:disability}.

\begin{figure}[!htbp]
	\centering
	\scalebox{0.88}{
	\begin{tikzpicture}[node distance=2em and 0em]
		\node[punktl, draw = none] (2) {};
		\node[punktl, left = 15mm of 2] (0) {active};
		\node[anchor=north east, at=(0.north east)]{$1$};
		\node[punktl, below = 15mm of 2] (3) {dead};
		\node[anchor=north east, at=(3.north east)]{$3$};
		\node[punktl, right =15mm of 2] (1) {disabled};
		\node[anchor=north east, at=(1.north east)]{$2$};
	\path
		(0)	edge [pil, bend left = 5]		node [above]		{$\lambda_{12}(t)$}				(1)
		(1)	edge [pil, bend left = 5]		node [below]		{\color{blue}$\lambda_{21}(t)$}				(0)
		(0)	edge [pil, bend right = 15]		node [below left]		{$\lambda_{13}(t)$}				(3)
		(1)	edge [pil, bend left = 15]		node [below right]		{$\lambda_{23}(t)$}				(3)

	;
	\end{tikzpicture}}
	\caption{The disability model.}
	\label{fig:disability}
\end{figure}

%
%
%
%
%

{

For simplicity, let us assume that $\dd B^{j,v}(s) = b^{j,v}(s) \, \dd s$. In this case, the approximation $\ap{C}_i(\dd s)$ of \eqref{cashflow_approximation} takes the form $\ap{C}_i(\dd s) = \ap{c}_i(s) \, \dd s$, where
\begin{align}\label{cashflow_approximation_density}  
\ap{c}_{i}(s) &= \sum_{j \in \mathcal{J}} \int_{0}^{s} \ap{p}_{ij}(s, \dd v) \, b^{j,v}(s) \nonumber \\
&\quad + \sum_{\substack{j,k \in \mathcal{J} \\ j \neq k}} \int_{0}^{s} \ap{p}_{i;jk}(s, \dd v) \, b^{(j,v),(k,0)}(s).
\end{align}
We refer to $\ap{c}_i$ as the approximated cashflow density.

}
 Now consider {initially the following payment functions}:
{
\begin{align*}
b^{j,v}(s)&= 1_{\{j=2\}} 1_{\{s\in [1,4]\}}\exp(-s/5)/100,\\
b^{(j,v),(k,0)}(s)&= 1_{\{j=1,k>1\}}+1_{\{j=2,k=3\}}.
\end{align*}}
The first one states that there are only stream payments when in disability, and they only commence after having been a minimum time of $1/10$ time units in disability.  { The second states that lump-sum payments (of unit size) are possible only when transitioning from active to disabled or dead or from disabled to dead. Notice that the above payment functions do not depend on $v$, which allows for straightforward implementation of Monte Carlo simulation methods, for effective comparison against our approximation. Later on, the payment functions are modified to depend on $v$ as well.}

Next, consider a semi-Markov structure for the underling jump-process, which is inspired by a calibration on real-life data from an undisclosed Danish life insurer. As stressed in the introduction, the fact that a Markov model is not sufficient to capture line insurance dynamics is well documented, and the internal calibration was done after exploring duration effects on the data with subsampling methods (cf. \cite{bladt2023conditional} and the Vignette in the \texttt{R} package \texttt{AalenJohansen} \cite{aajo} for statistical non-parametric unveiling of duration effects from raw data). After unveiling non-Markovian effects, the final estimation is parametric, and uses a semi-Markov spline parametrization for the transition rates of the general likelihood formula of a jump-process, as provided in \cite{andersen2012statistical}.

The transition rates are given as follows. The transition from active to death, given by $\lambda_{13}(t,u)$, is independent of duration and calibrated to mortality data through spline interpolation as shown in the top left panel of Figure \ref{fig:rates}.
\begin{figure}[!htbp]
\centering
\includegraphics[clip, trim=0cm 0cm 0cm 0cm,width=0.49\textwidth]{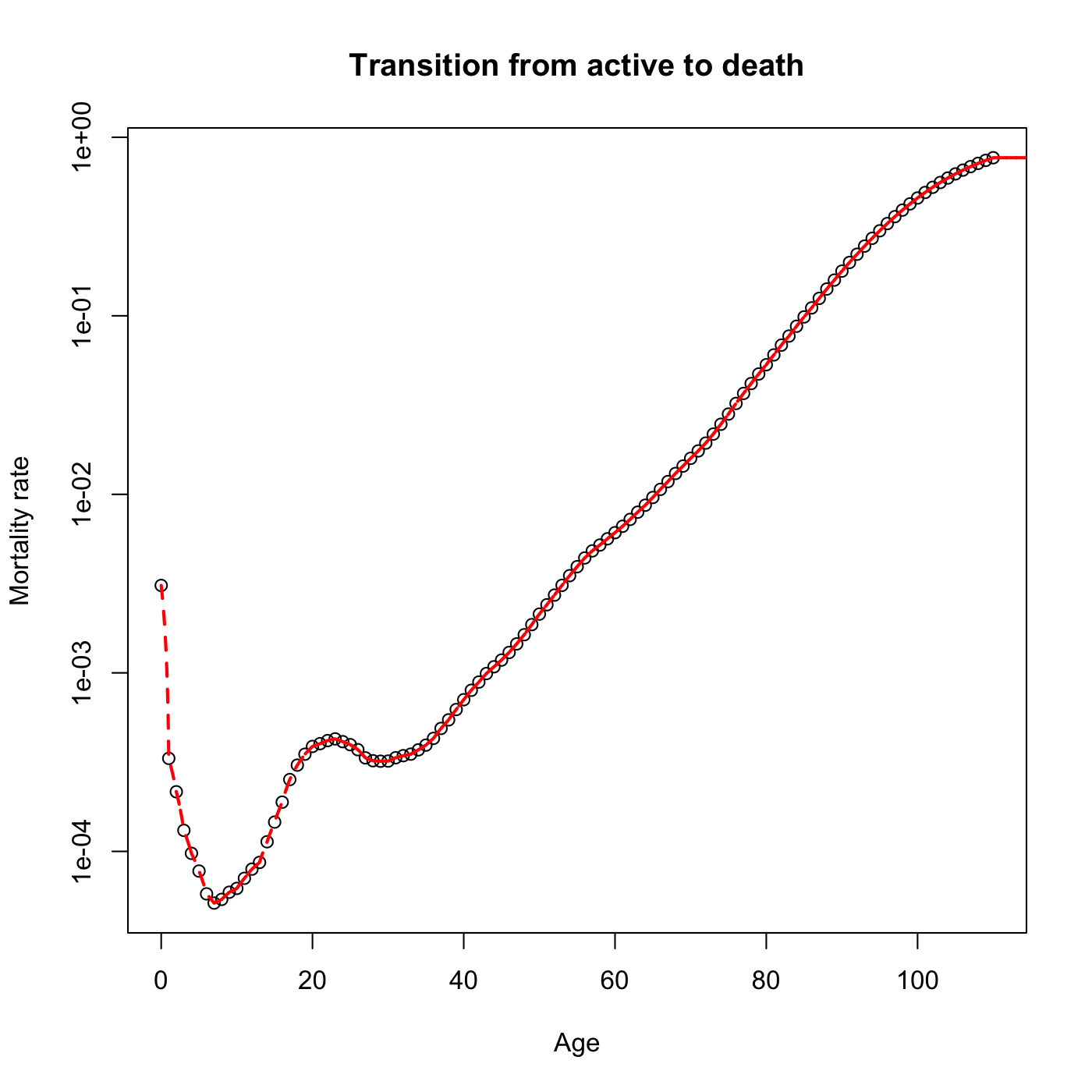}
\includegraphics[clip, trim=0cm 0cm 0cm 0cm,width=0.49\textwidth]{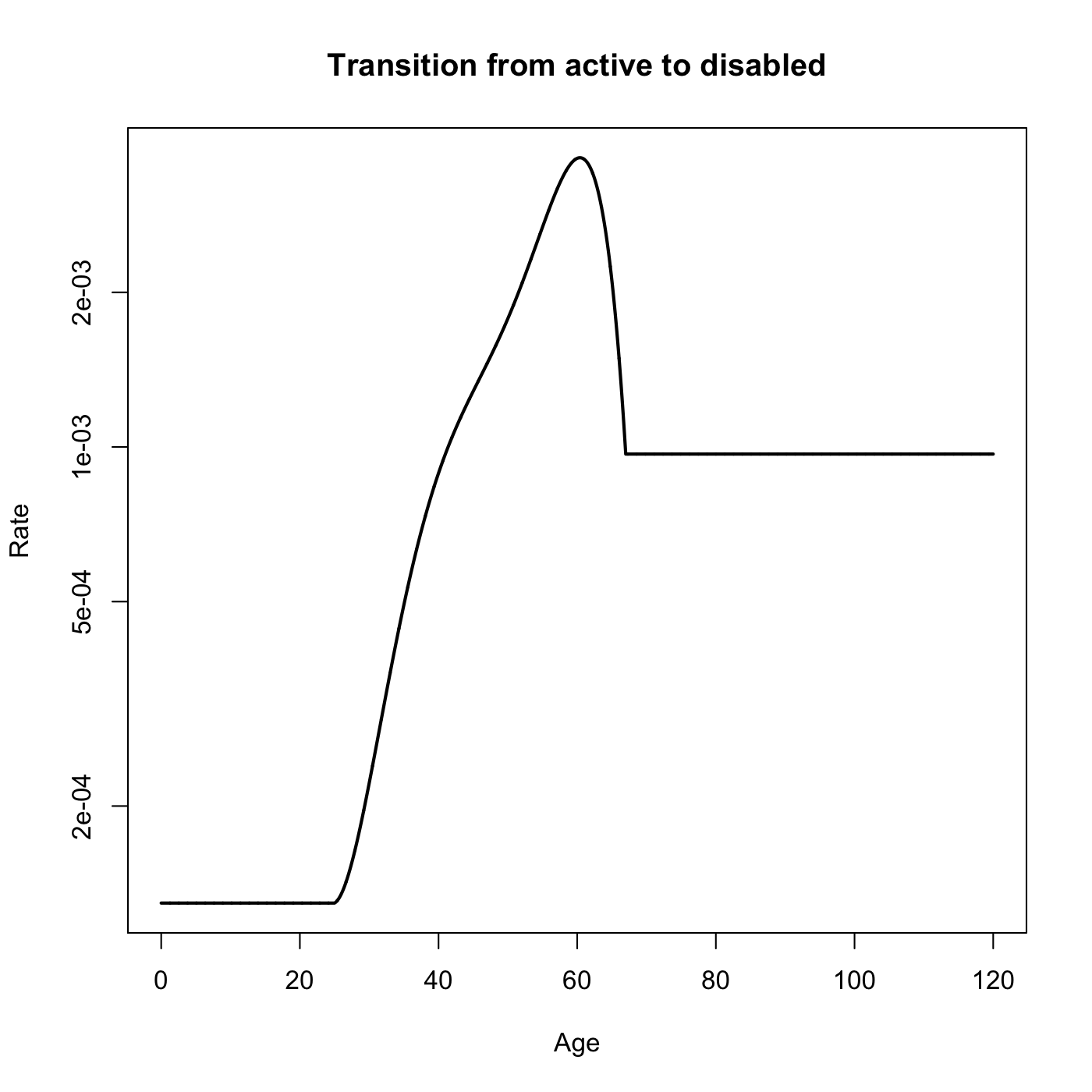}
\includegraphics[clip, trim=0cm 0cm 0cm 0cm,width=0.49\textwidth]{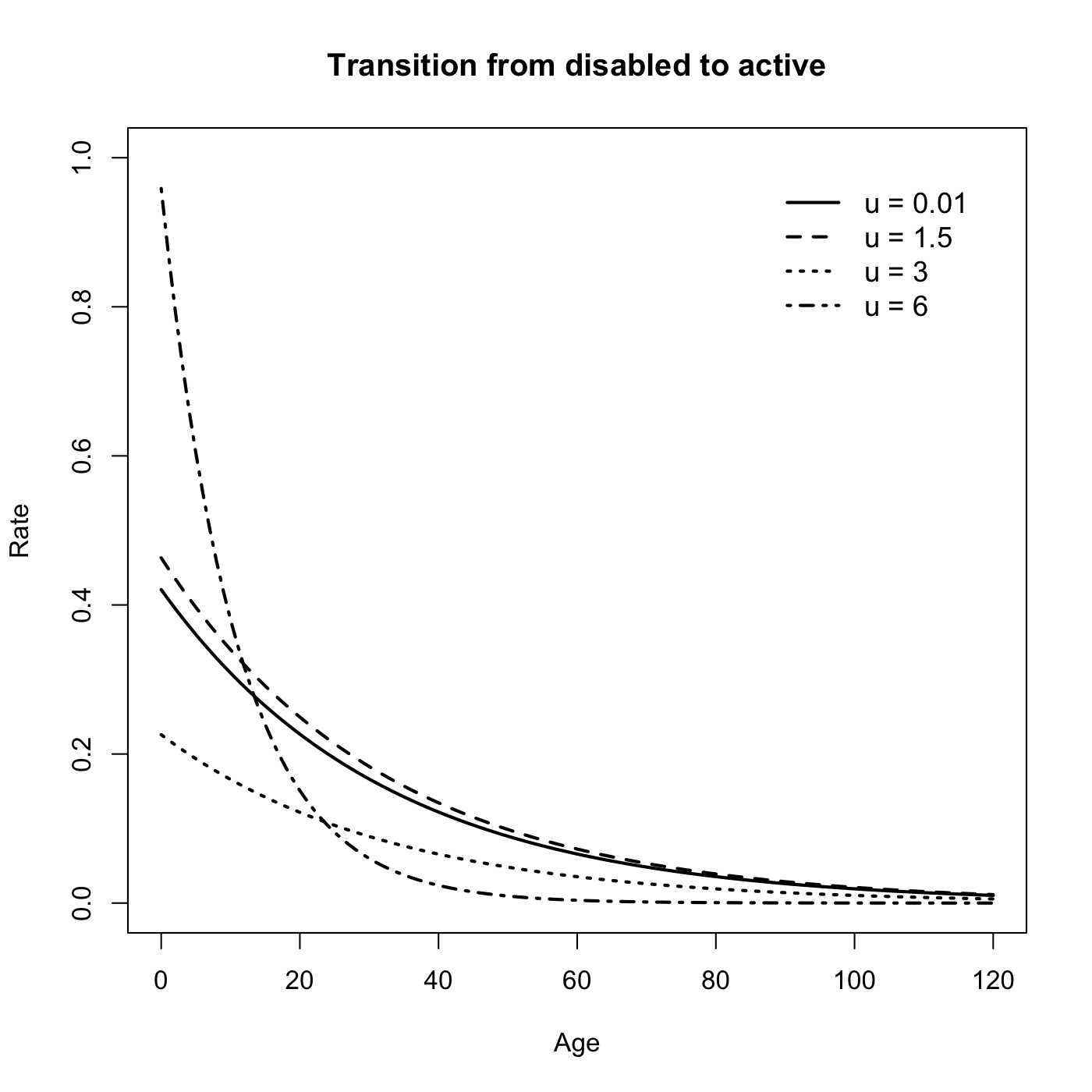}
\includegraphics[clip, trim=0cm 0cm 0cm 0cm,width=0.49\textwidth]{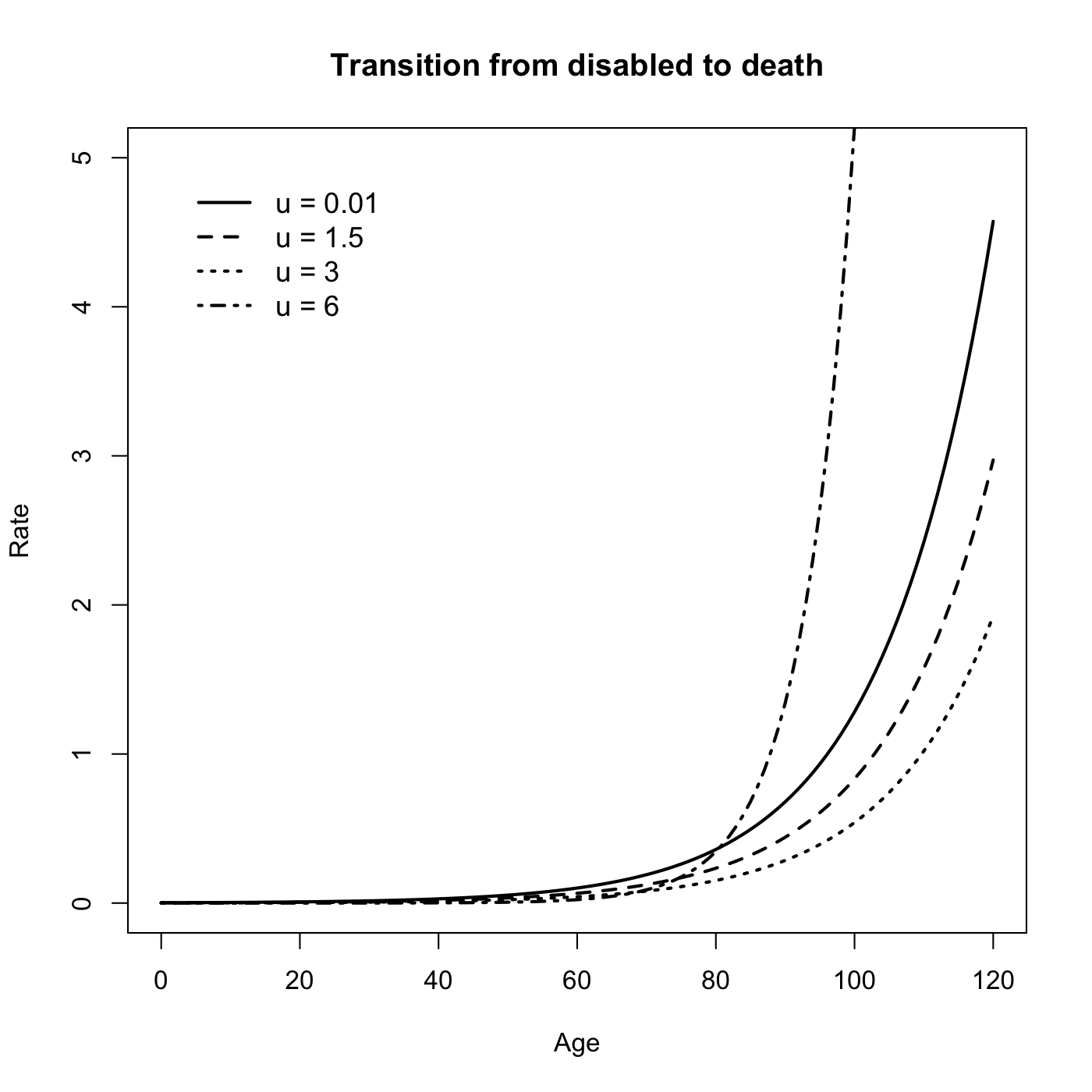}
\caption{Semi-Markov transition rates for life-insurance portfolio. Top left: mortality data and calibrated splines; top right: transition rate from active to disabled; bottom left: transitions from disabled to active; bottom right: transitions from disabled to death.
} \label{fig:rates}
\end{figure}

Transitions to disability are also duration independent, and calibrated internally, which provides the following estimate:
\begin{align*}
\lambda_{12}(t,u)= \begin{cases}
  \lambda_{12,\text{base}}(25) & \text{if } t \leq 25 \\
  \lambda_{12,\text{base}}(t) & \text{if } 25 < t \leq 67 \\
  \lambda_{12,\text{base}}(67) & \text{if } t > 67,
\end{cases}
\end{align*}
where $ \lambda_{12,\text{base}}(t)$ is the exponential of a fifth-order polynomial in $t$. A figure of $\lambda_{12}(t,u)$ as a function of $t$ (since it does not depend on $u$) is provided in the top right panel of Figure \ref{fig:rates}. Notice that before entering the workforce and after retirement age, the disability rate is constant, according to this model. The recovery rate is given by
\begin{align*}
\lambda_{21}(t,u)=
\begin{cases}
  \exp(\phi_3 + \beta_1 t + \theta_3 u) & \text{if } u \leq 0.23 \\
  \exp(\phi_2 + \beta_1 t + \theta_2 u) & \text{if } 0.23 < u \leq 2 \\
  \exp(\phi_1 + \beta_1 t + \theta_1 u) & \text{if } 2 < u \leq 5 \\
  \exp(\phi_0 + \beta_2 t) & \text{if } u > 5,
\end{cases}
\end{align*}
for suitable parameters $\phi_i,\beta_i,\theta_i$. The bottom left panel of Figure \ref{fig:rates} depicts the behavior of the recovery rate as a function of $t$ for various fixed duration values $u$. We see that recovery does not improve with age, and duration also impacts recovery negatively for adults. Realistically, the rate is only precise for ages $25$ to $65$ since that is the range where most data lies in. Extrapolation into very young or very old ages is unreliable but also usually not required for practical purposes.

The only other rate that requires specification is  {the} mortality rate from disability, given by
\begin{align*}
\lambda_{21}(t,u)=
\begin{cases}
  \exp(\alpha_{1} + \eta_{1} t + \zeta_{1} u) & \text{if } u \leq 5 \\
  \exp(\alpha_{2} + \eta_{2} t) & \text{if } u > 5
\end{cases}
\end{align*}
for suitable $\alpha_i,\eta_i,\zeta_i$, and is depicted in the bottom right panel of Figure \ref{fig:rates}.

 {First, we compute the cashflow for the disabled state, for $u=0$, by applying our approximation \eqref{cashflow_approximation}. Here, Theorem \ref{conditional_transitions_approx} provides the required expression for the transition probability approximation, while the integrals were computed numerically. The infinite representation was truncated at $15,50,150$ terms, and the precision was respectively $\gamma=3,10,30$, which upon division guarantees a good approximation up to time $5$, beyond which cashflows are very small. The latter parameters were chosen to match the computation time in a local implementation of (importance-sampled) Monte Carlo simulation of the cashflow for $N_{\mbox{sim}}=10^3,10^4,10^5$ simulated paths, respectively. The result of the approximation and the Monte Carlo simulation is shown in Figure \ref{fig:cashflow}. We notice that our implementation provides a much more stable cashflow approximation even for $\gamma=10$, while $\gamma=30$ can practically be taken as the true cashflow. The Monte Carlo implementation is much slower to converge, exhibiting noise even for $N_{\mbox{sim}}=10^5$ simulations. Both approximations were evaluated using high-level programming languages, and it is arguable that the recipe-like strong approximation is more straightforward and less prone to human error. {Finally, when considering $u> 0$, both approximations can applied on an augmented state-space, and thus similar conclusions to the $u=0$ case apply in those settings.}
\begin{figure}[!htbp]
\centering
\includegraphics[clip, trim=0cm 0cm 0cm 0cm,width=1\textwidth]{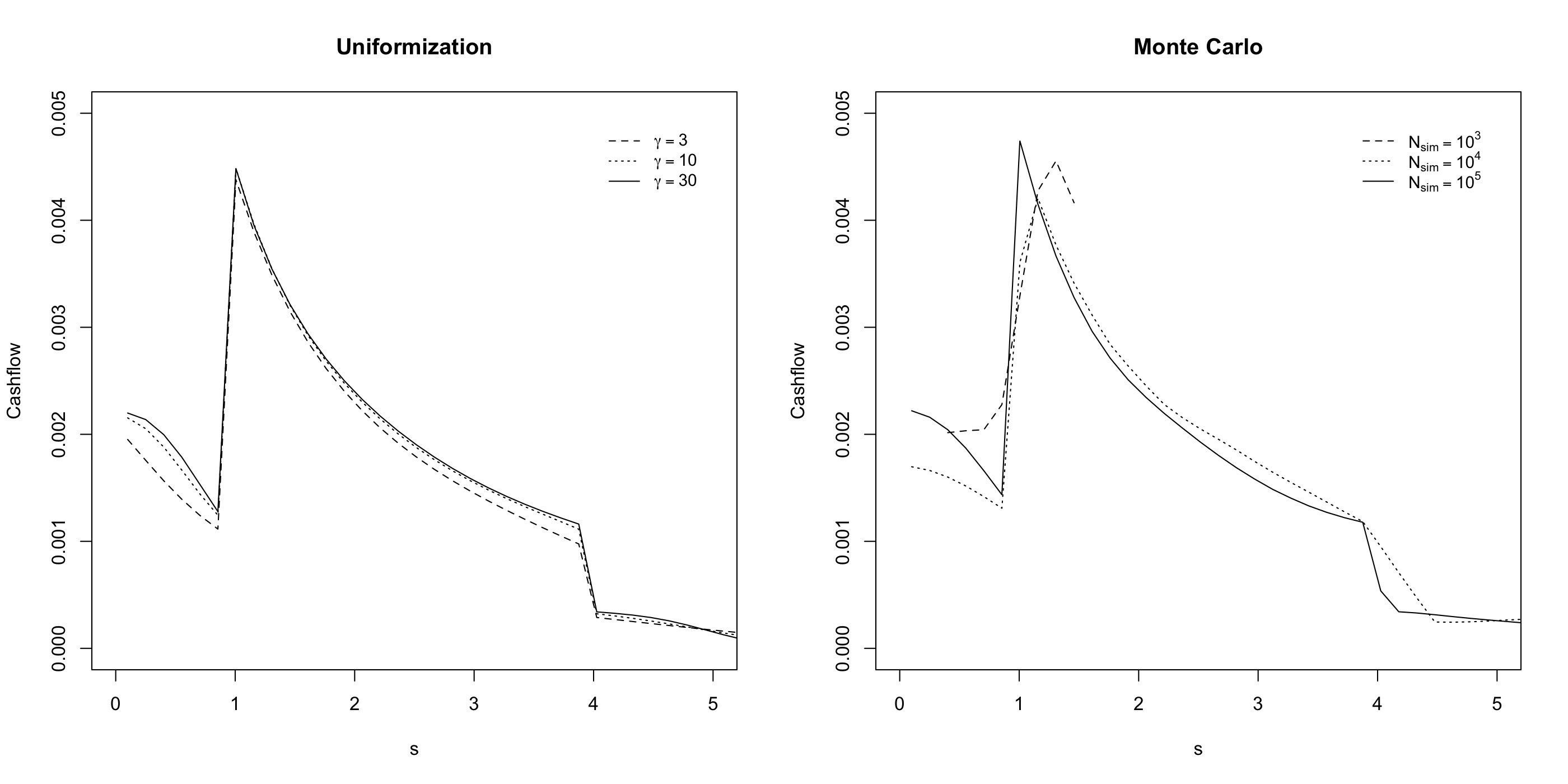}
\caption{{Uniformization (left) and Monte Carlo (right) approximations of the cashflow \eqref{cashflow_definition} for the disabled state, for varying precisions and comparable computation times. The payment functions are not duration-dependent. }
} \label{fig:cashflow}
\end{figure}

Next, we keep the same semi-Markov parameters but consider a payment function that depends on $v$, namely
\begin{align*}
b^{j,v}(s)&= 1_{\{j=2\}} 1_{\{v\ge1/10\}} 1_{\{s\in [1,4]\}}\exp(-(s-v)/5)/100,\\
b^{(j,v),(k,0)}(s)&= 1_{\{j=1,k>1\}}+1_{\{j=2,k=3\}}.
\end{align*}
All other parameters are kept equal, except that we now consider the $u=0$ and $u=1$ cases. Naturally, we expect an increase in cashflow for the latter case. Practically speaking, we use the state-augmentation technique to deal with the $u=1$ case, with intensity matrix now given by \eqref{state_augmentation}. An alternative is using an integro-differential equation solver as benchmark, though it is outside of the scope of this paper. The results are provided in Figure \ref{fig:cashflow2}, where convergence concerning $\gamma$ is again observed.
\begin{figure}[!htbp]
\centering
\includegraphics[clip, trim=0cm 0cm 0cm 0cm,width=1\textwidth]{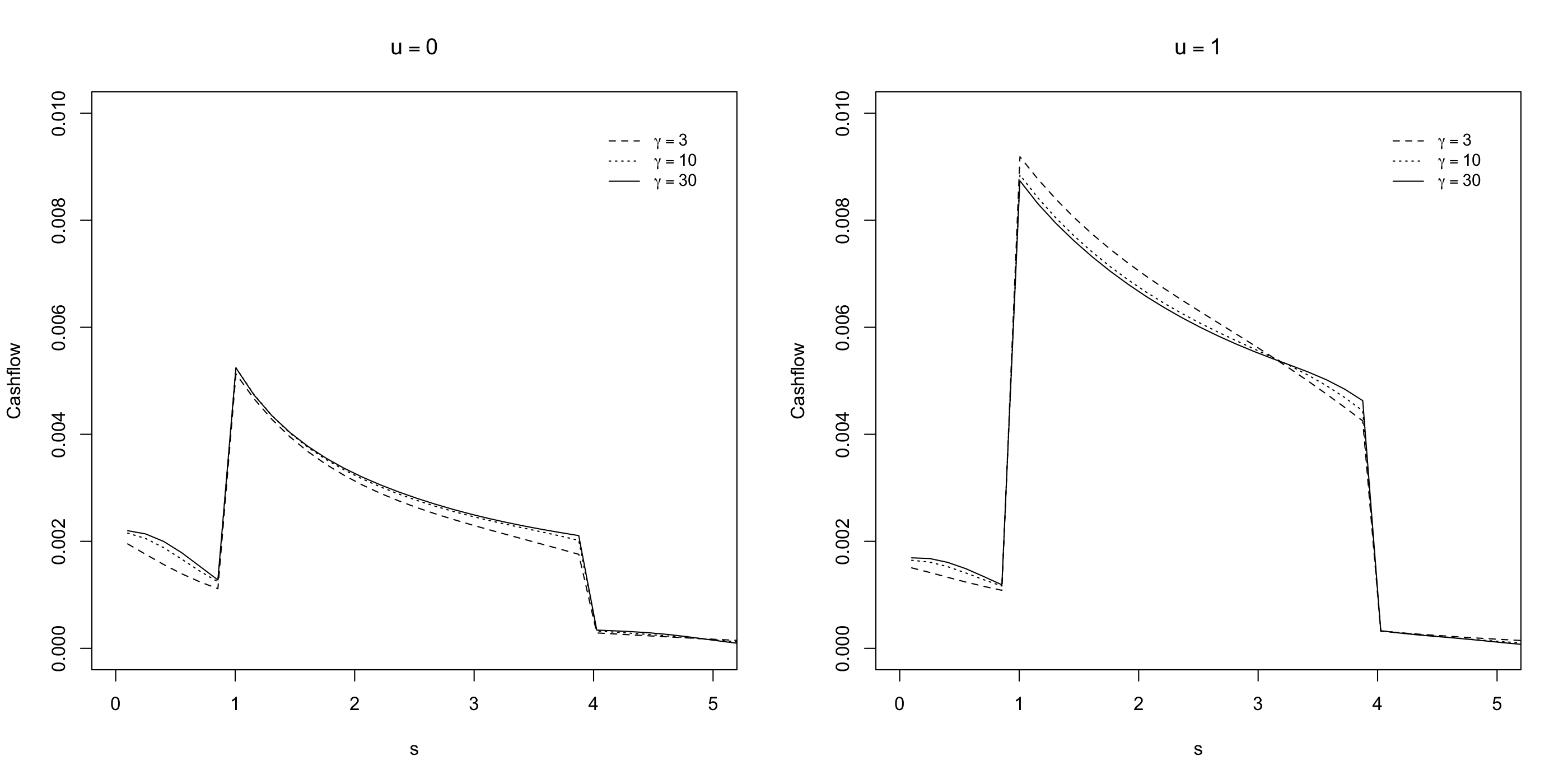}
\caption{{Approximated cashflow density \eqref{cashflow_approximation_density} for the disabled state, for varying precisions. The cases $u=0$ (left) and $u=1$ (right) are considered. The payment functions are duration-dependent. }
} \label{fig:cashflow2}
\end{figure}
 }

\subsection{Reduced form models for credit risk}
\label{sec:credit}
{At least since \cite{jarrow1997markov}, there has been a rich exchange on multi-state modeling between life insurance and credit risk.}
In this section we suggest how our model could be extended for pricing defaultable claims (such as defaultable bonds, or Credit Default Swaps) in reduced form models.

The classical martingale approach to asset pricing  consists of taking an expectation under a risk neutral measure $\mathbb Q$, equivalent to the real-world  probability measure $\mathbb P$, of a 
contingent claim, and discounted by the cumulative interest.
When the claim is defaultable, in  other words, there is a risk that it is not paid in full, then it was shown in \cite{duffie1999modeling, collin2004general} that valuation can proceed similarly, albeit with a discount factor that is default adjusted and under a risk neutral measure that is absolutely continuous with respect to $\mathbb P$ and not necessarily equivalent.

Consider a claim that pays a contingent payoff $X(T)$ at maturity $T$ in the case of no default. The default can be generic, either the default of the company as in the case of a defaultable  bond, or it can be the default of a reference entity, in the case of Credit Default Swaps. Since the seminal paper of \cite{duffie1999modeling}, closed form solutions can be provided under an assumption of  zero recovery as
\begin{equation}
V(t) = \E^{\mathbb Q} \left[ e^{-\int_t^T (r_s + h_s) ds} X(T) |  \mathcal{F}_t \right],
\end{equation}
where $r_t$ represents an interest rate and $h_t$ represents the default rate. This setup can be extended to case of a recovery  that is a proportional to the market value, but for simplicity, we consider here the case of zero recovery rate.
Our framework allows us to extend the classical valuation to the case of semi-Markov state-variable processes. 
As in the insurance setting, the credit risk model is driven by the state-variable process $Z=\{Z(t)\}_{t\ge 0}$ with duration process  $U=\{U(t)\}_{0\le t\le T}$.

As in the classical credit risk setup of  \cite{duffie1999modeling}, we let $X(T) = \Phi(Z(T))$, for some function $\Phi$. The price at time $t$ of the defaultable claim becomes
\[ V(t;i,u)=\E^{\mathbb Q}\left[e^{-\int_t^T (r_s + h_s) ds } \Phi(Z(T)) \mid Z(t)=i,U(t)=u\right],\quad i\in\mathcal{J},\: u\ge0, \]

It is understood that in the semi-Markovian setup, the default  rate $h_t$ depends on the duration as well. We can think it $h_t$ as an absorbing intensity, i.e., the transition rate to a cemetery state $\partial$ that represents the state of default,
\begin{equation}
h_t = \lambda_{Z(t), \partial}(t, U(t)).
\end{equation} 

{
Note that this is a particular case of the general formulation in Subsection 1.4.2  in \cite{duffie1999modeling}. However, for valuation, these authors propose an affine defaultable term-structure models with jumps. See also, e.g., \cite{duffie2004credit} for a comprehensive review of affine models.}
Our semi-Markov setup allows for an alternative valuation setup. The hazard function $\lambda$ is an arbitrary  function that verifies the boundedness condition (for the grid-conditional approximation) and the Lipschitz continuity condition for the unconditional approximation. 
Note that the discount factor depends on the entire trajectory of the intensity process. We conjecture that our baseline setup can be extended to this type of valuation in order to achieve  computational tractability.

\subsection{Control}
\label{sec:control}
We now assume that  that an agent can choose, at time $0$, a control that depends on the time and duration process which affects the process $Z$ via its intensities. 
Under mild assumptions, our framework allows for such control. More specifically, for a set $\mathcal{A}$ of real measurable functions over $\mathds{R}_+^2$, the agent chooses some $a \in \mathcal{A}$ such that $Z$ is driven by $\bm{\Lambda}(s,v, a(s,v))$. For a fixed $a$, $\bm{\Lambda}(\cdot,\cdot,\cdot)$ is effectively only dependent on $s$ and $v$. The construction of such a process remains identical to that of Section \ref{sec:SMJP}. In fact, it results in a construction on the same probability space for all $a \in \mathcal{A}$. The conditional approximation remains exactly the same as well. For the unconditional version,  the following conditions are necessary in order to retain the entry-wise Lipschitzness of intensity matrix

\begin{itemize}
       \item $\Lambda$ is bounded and Lipschitz continuous on its three entries,
       \item The family $\mathcal{A}$ is uniformly bounded and uniformly Lipschitz continuous (with every element sharing the same Lipschitz constant).
\end{itemize}
Moreover, under such conditions, the family $\mathcal{A}$ is compact, ensuring the existence of an optimal control at time zero. This paves the way for efficient computations of approximations for value functions and optimization over control policies.

It would be interesting to extend our setup to that of an intensity control problem, as presented in \cite{bismut1975controle,bremaud1981point}. In this scenario, the goal is to minimize a criterion of the form:

$$\E^{\mathbb Q}\left[ \int_0^T e^{-\int_0^t \rho_s ds } \phi(t, Z(t), U(t), \Lambda(t))dt + e^{-\int_t^T \rho_s ds}\Phi(T, Z_T)\right],$$
for a running cost $\phi$ and a terminal cost $\Phi$.
When  the discount factor is given as a  default adjusted discount of  the previous section  $\rho_t =  \lambda_{Z(t), \partial}(t, U(t))$, this setup no longer falls under our framework and would necessitate an extension to discount factors that depend on the trajectory of the intensity.
We further assume that the running cost depends on the jump intensity from the current state, i.e, 
$\phi(t, Z(t), U(t), \Lambda(t)) = \phi(t, Z(t), U(t), \Lambda_{Z(t), \cdot} (t, U(t)))$ and we assume that $\phi$ is a strictly convex running cost.
The optimal $V^*$ solves the following Hamilton-Jacobi-Bellman Equation

\begin{align*}
0 =& \frac{\partial V^*(t, i, u) }{\partial t} + \frac{\partial V^*(t, i, u) }{\partial u} \\
&+  \inf_{\Lambda} (- \Lambda_{i, \partial} (t, u) V^*(t, i, u)  + \sum_{j}\Lambda_{i,j} (t,u) [V^*(t, j, u) - V^*(t, i, u)]
  + \phi(t, i, u, \Lambda)),\\
&V^*(T, i, u) = \Phi (T, i).
\end{align*}
By discretizing time and using approximations of the value function (which can be computed at a given time $t$ for all $j$ and $u$ at once),  the optimal solution $\Lambda^*(t, u)$ could  be approximated by solving 
$$\inf_{\Lambda} (- \Lambda_{i, \partial} (t, u)) \mathbb V^*(t, i, u)  + \sum_{j}\Lambda_{i,j} (t,u) [\mathbb V^*(t, j, u) - \mathbb V^*(t, i, u)]
  + \phi(t, i, u, \Lambda_{i, \cdot} (t, u))), $$
where we propose to replace the exact optimal solution $V^*$ by its approximation.

Note that this is a constrained optimization problem, since $\sum_j \Lambda_{i,j} = 0$. Given the linear constraint, the problem has a unique solution for a variety of running cost functions, for example quadratic costs.
We leave for future work the extension of our setup to this case, where there is a dependence of discount factors on the trajectory of the intensity process.
\appendix

\section{Proof of Theorem \ref{th:strong_simp}}\label{sec:strongSMJP}
We build on the the proof for strong convergence introduced in \cite{bladt_peralta}   for the simpler case of time-inhomogeneous Markov jump processes. Since the semi-Markov case features duration dependence, the challenge will be to carefully account for the successive epochs at which the duration process returns to zero.  Let us borrow the notation and concepts introduced in Section \ref{sec:approxgen1}, here shown superindexed by the term $(\gamma)$ to make explicit their dependence on $\gamma$. In order to prove convergence in the $J_1$-topology over the space of $\mathcal{J}$-valued c\`adl\`ag functions, denoted by $\mathcal{D}(\mathds{R}_+,\mathcal{J})$, it is sufficient to find a family $\{\Delta^{(\gamma)}\}_{\gamma}$ of random homoeomorphic functions over $\mathds{R}^+$, and an increasing sequence $\{c_k\}_{k\ge 1}$ with $c_k\uparrow\infty$, such that
\begin{align}
\sup_{s\ge 0}|\Delta^{(\gamma)}(s) - s|&\rightarrow 0,\label{eq:strong4}\\
\sum_{k=1}^{\infty}2^{-k}\times\mathds{1} \left\{\mathbb{Z}^{(\gamma)}\big(\Delta^{(\gamma)}(s)\big)\neq\, Z(s)\mbox{ for some }s\le c_k\right\}&\rightarrow 0,\label{eq:strong5}\\
\sum_{k=1}^{\infty}2^{-k}\left(\left(\sup_{s\le c_k}\left|\mathbb{U}^{(\gamma)}\big(\Delta^{(\gamma)}(s)\big)- U(s)\right|\right)\wedge 1\right)&\rightarrow 0,\label{eq:strong6}
\end{align}
in an almost sure sense as $\gamma\rightarrow\infty$. Heuristically speaking, the function $\Delta^{(\gamma)}$ acts as a time change that for increasing compact time intervals, ``couples'' the paths of $\mathbb{Z}^{(\gamma)}$ and $Z$, approximates those of $\mathbb{U}^{(\gamma)}$ to $U$, while $\Delta^{(\gamma)}$ converges uniformly to the identity function. In fact, in order to prove (\ref{eq:strong4})-(\ref{eq:strong6}) for our choice of $\Delta^{(\gamma)}$, next we obtain explicit bounds for their rates of strong convergence.

Fix $\varepsilon>0$ and define
\begin{align}
\Delta^{(\gamma)}(t)=\left\{\begin{array}{ccc}  t\frac{\theta^{(\gamma)}_1}{\chi^{(\gamma)}_1}&\mbox{for}& t\in \left[0,\chi^{(\gamma)}_1\right)\\
\theta^{(\gamma)}_\ell + (t-\chi^{(\gamma)}_\ell)\frac{\theta^{(\gamma)}_{\ell+1}-\theta^{(\gamma)}_{\ell}}{\chi^{(\gamma)}_{\ell+1}-\chi^{(\gamma)}_{\ell}}&\mbox{for}& t\in \left[\chi^{(\gamma)}_\ell, \chi^{(\gamma)}_{\ell+1}\right),\,\ell<\lfloor \gamma^{1+\varepsilon}\rfloor\\
\theta^{(\gamma)}_{\lfloor \gamma^{1+\varepsilon}\rfloor} + (t-\chi^{(\gamma)}_{\lfloor \gamma^{1+\varepsilon}\rfloor})&\mbox{for}& t\in \left[\chi^{(\gamma)}_{\lfloor \gamma^{1+\varepsilon}\rfloor}, \infty\right).
\end{array}\right.\label{eq:defDelta1}
\end{align}
The function $\Delta$ is, in essence, a piecewise linear function that maps the point $\chi^{(\gamma)}_\ell$ to $\theta^{(\gamma)}_\ell$ for all $\ell\le \lfloor \gamma^{1+\varepsilon}\rfloor$, and increases linearly at a rate $1$ after $\chi^{(\gamma)}_{\lfloor \gamma^{1+\varepsilon}\rfloor}$. Particularly, this implies that 
\begin{equation}\label{eq:timechange1}
\mathbb{Z}^{(\gamma)}(\Delta^{(\gamma)}(s)) = Z(s) \mbox{ for all } s\in \left[0, \chi^{(\gamma)}_{\lfloor \gamma^{1+\varepsilon}\rfloor}\right).\end{equation} Furthermore, it can be readily verified that 
\begin{equation}\label{eq:rateDelta6}
\sup_{s\ge 0}|\Delta^{(\gamma)}(s) - s| = \max_{\ell\in\{1,2,\dots, \lfloor \gamma^{1+\varepsilon}\rfloor\} } \left|\chi^{(\gamma)}_\ell - \theta^{(\gamma)}_\ell\right|.
\end{equation}
Employing \cite[Lemma 1]{bladt_peralta}, we arrive {at}
\begin{equation}
\mathds{P}\left(\sup_{s\ge 0}|\Delta^{(\gamma)}(s) - s| \ge \alpha (\log \gamma) \gamma^{-1/2 + \varepsilon/2}\right) = o(\gamma^q),\label{eq:rateDelta5}
\end{equation}
where $q$ is any arbitrary positive real number and $\alpha$ is a constant that only depends on $\varepsilon$ and $q$. Via standard Borell-Cantelli arguments (e.g., choosing $q=2$), it is straightforward to get \eqref{eq:strong4} from \eqref{eq:rateDelta6}. 

For \eqref{eq:strong5}, note that due to \eqref{eq:timechange1},
\begin{align}
&\sum_{k=1}^{\infty}2^{-k}\times\mathds{1} \left\{\mathbb{Z}^{(\gamma)}\big(\Delta^{(\gamma)}(s)\big)\neq\, Z(s)\mbox{ for some }s\le c_k\right\}\nonumber\\
&\quad \le \sum_{k=1}^{\infty}2^{-k}\times\mathds{1} \left\{\chi^{(\gamma)}_{\lfloor \gamma^{1+\varepsilon}\rfloor}\le c_k\right\}\nonumber\\
&\quad = \sum_{k=1}^{\lfloor \gamma^{\varepsilon}\rfloor}2^{-k}\times\mathds{1} \left\{\chi^{(\gamma)}_{\lfloor \gamma^{1+\varepsilon}\rfloor}\le c_k\right\} + \sum_{k=\lfloor \gamma^{\varepsilon}\rfloor+1}^{\infty}2^{-k}\times\mathds{1} \left\{\chi^{(\gamma)}_{\lfloor \gamma^{1+\varepsilon}\rfloor}\le c_k\right\}\nonumber\\
& \quad\le \mathds{1} \left\{\chi^{(\gamma)}_{\lfloor \gamma^{1+\varepsilon}\rfloor} \le c_{\lfloor \gamma^{\varepsilon}\rfloor}\right\} + 2^{-\lfloor \gamma^{\varepsilon}\rfloor},\label{eq:auxbound5}
\end{align}
where in the last inequality we used that $\{c_k\}_k$ is a (yet to be determined) increasing sequence. Since $2^{-\lfloor \gamma^{\varepsilon}\rfloor}\rightarrow 0$, then \eqref{eq:strong5} will follow from \eqref{eq:auxbound5} once we find a sequence $\{c_k\}_k$ such that $\mathds{1} \left\{\chi^{(\gamma)}_{\lfloor \gamma^{1+\varepsilon}\rfloor} \le c_{\lfloor \gamma^{\varepsilon}\rfloor}\right\}\rightarrow 0$ almost surely. We claim that $c_k=k/2$ exhibits the desired property. Indeed, for this choice of $c_k$,
\begin{align}
\left\{\chi^{(\gamma)}_{\lfloor \gamma^{1+\varepsilon}\rfloor} \le c_{\lfloor \gamma^{\varepsilon}\rfloor}\right\} = \left\{\chi^{(\gamma)}_{\lfloor \gamma^{1+\varepsilon}\rfloor} \le \frac{\lfloor \gamma^{\varepsilon}\rfloor}{2}\right\} = \left\{\frac{\gamma}{\lfloor \gamma^{1+\varepsilon}\rfloor}\chi^{(\gamma)}_{\lfloor \gamma^{1+\varepsilon}\rfloor} \le \frac{\gamma}{\lfloor \gamma^{1+\varepsilon}\rfloor}\frac{\lfloor \gamma^{\varepsilon}\rfloor}{2}\right\}.\label{eq:auxsets5}
\end{align}
The term $\frac{\gamma}{\lfloor \gamma^{1+\varepsilon}\rfloor}\chi^{(\gamma)}_{\lfloor \gamma^{1+\varepsilon}\rfloor}$ converges to 1 almost surely. This convergence can be verified by interpreting the term as the sum of $\lfloor \gamma^{1+\varepsilon}\rfloor$ i.i.d. exponential r.v.'s of parameter 1 scaled by the number of summands and then applying the SLLN. Additionally, the term $\frac{\gamma}{\lfloor \gamma^{1+\varepsilon}\rfloor}\frac{\lfloor \gamma^{\varepsilon}\rfloor}{2}$ converges to $\frac{1}{2}$. Given these convergences, \eqref{eq:auxsets5} represents an asymptotically null set, and thus, \eqref{eq:strong5} follows.

Finally, for \eqref{eq:strong6}, employing similar inequalities that lead to \eqref{eq:auxbound5}, we get
\begin{align}
&\sum_{k=1}^{\infty}2^{-k}\left(\left(\sup_{s\le c_k}\left|\mathbb{U}^{(\gamma)}\big(\Delta^{(\gamma)}(s)\big)- U(s)\right|\right)\wedge 1\right)\nonumber\\
& \quad\le \left(\left(\sup_{s\le c_{\lfloor \gamma^{\varepsilon}\rfloor}}\left|\mathbb{U}^{(\gamma)}\big(\Delta^{(\gamma)}(s)\big)- U(s)\right|\right)\wedge 1\right) + 2^{-\lfloor \gamma^{\varepsilon}\rfloor}\nonumber\\
& \quad\le \left(\left(\sup_{s < \chi^{(\gamma)}_{\lfloor \gamma^{1+\varepsilon}\rfloor}}\left|\mathbb{U}^{(\gamma)}\big(\Delta^{(\gamma)}(s)\big)- U(s)\right|\right)\wedge 1\right) + \mathds{1} \left\{\chi^{(\gamma)}_{\lfloor \gamma^{1+\varepsilon}\rfloor} \le c_{\lfloor \gamma^{\varepsilon}\rfloor}\right\} + 2^{-\lfloor \gamma^{\varepsilon}\rfloor}.\label{eq:auxound6}
\end{align}
The almost sure convergence of the last two summands of \eqref{eq:auxound6} to zero were confirmed via \eqref{eq:auxsets5}. What remains is the almost sure convergence of the first summand of \eqref{eq:auxound6} to zero: establishing this will prove \eqref{eq:strong6}. 

Let $\{\sigma^{(\gamma)}_\ell\}_{\ell}$ be the successive epochs at which $U$ returns to $0$. Note that $U$ evolves like the identity function with a space translation to $0$ at the time epochs $\{\sigma^{(\gamma)}_\ell\}_{\ell}$. Similarly, on the time interval $[0, \chi^{(\gamma)}_{\lfloor \gamma^{1+\varepsilon}\rfloor})$, $\mathbb{U}^{(\gamma)}\circ\Delta^{(\gamma)}$ evolves $\Delta^{(\gamma)}$ with a space-translation to $0$ at the time epochs $\{\sigma^{(\gamma)}_\ell\}_{\ell}$. Thus, for $s\in [\sigma^{(\gamma)}_\ell, \sigma^{(\gamma)}_{\ell+1})$, 
\begin{align*}
&\left|\mathbb{U}^{(\gamma)}\big(\Delta^{(\gamma)}(s)\big)- U(s)\right|\\
 &\quad = \left|\mathbb{U}^{(\gamma)}\big(\Delta^{(\gamma)}(s)\big)- U(s) + \left(\Delta^{(\gamma)}(\sigma^{(\gamma)}_\ell) - \sigma^{(\gamma)}_\ell\right) - \left(\Delta^{(\gamma)}(\sigma^{(\gamma)}_\ell) - \sigma^{(\gamma)}_\ell\right)\right|\\
 &\quad \le \left|\left(\mathbb{U}^{(\gamma)}\big(\Delta^{(\gamma)}(s)\big)+ \Delta^{(\gamma)}(\sigma^{(\gamma)}_\ell)\right)- \left(U(s) +  \sigma^{(\gamma)}_\ell\right)\right| + \left|\Delta^{(\gamma)}(\sigma^{(\gamma)}_\ell) - \sigma^{(\gamma)}_\ell\right|\\
 &\quad = \left|\Delta^{(\gamma)}(s) - s\right| + \left|\Delta^{(\gamma)}(\sigma^{(\gamma)}_\ell) - \sigma^{(\gamma)}_\ell\right|.
\end{align*}
From this, it follows
\begin{align}
\sup_{s < \chi^{(\gamma)}_{\lfloor \gamma^{1+\varepsilon}\rfloor}}\left|\mathbb{U}^{(\gamma)}\big(\Delta^{(\gamma)}(s)\big)- U(s)\right|\le 2 \times \sup_{s\ge 0}\left|\Delta^{(\gamma)}(s)-s\right|.\label{eq:auxdurationbound1}
\end{align}
This asserts the almost sure convergence of the l.h.s. of \eqref{eq:auxdurationbound1} to $0$, as supported by \eqref{eq:strong4}. Consequently, \eqref{eq:strong6} is derived, completing the proof.
\section{Proof of Theorem \ref{th:convergence_densities}}\label{sec:convergence_densities}
Before we prove Theorem \ref{th:convergence_densities}, we first present a couple of technical matrix-analytic results, namely, Lemma \ref{lem:boundPi} and Lemma \ref{lem:boundCk} below.

For any real matrices or vectors $\bm{A}^{(1)}=\{A_{ij}^{(1)}\}_{i,j}$ and $\bm{A}^{(2)}=\{A_{ij}^{(2)}\}_{i,j}$ of the same dimensions, we denote by $\bm{A}^{(1)}\le \bm{A}^{(2)}$ the case for which $A_{ij}^{(1)}\le A_{ij}^{(2)}$ for all $i,j$.
Moreover, for a real matrix or vector $\bm{A}=\{A_{ij}\}_{i,j}$, define $|\bm{A}|$ to be the matrix whose entries correspond to the absolute values of those of $\bm{A}$. In particular, if $\bm{e}$ is a column-vector of ones of appropriate size, then
\[(|\bm{A}|\bm{e})_j=\sum_k |A_{jk}|\quad\text{for all } j.\]
Essentially, $(|\bm{A}|\bm{e})_i$ corresponds to { the total variation of the $i$-th row of $\bm{A}$, fact that will aid us in providing a more general proof in Appendix \ref{sec:GeneralProofs}}. Some straightforward properties follow:
\begin{itemize}
       \item $\bm{0}\le |\bm{A}|\bm{e}$ where $\bm{0}$ is a column-vector of zeroes of appropriate size,
       \item $|\bm{A}^{(1)} + \bm{A}^{(2)}|\bm{e}\le |\bm{A}^{(1)}|\bm{e} + |\bm{A}^{(2)}|\bm{e}$ for all matrices $\bm{A}^{(1)}$ and $\bm{A}^{(2)}$ of the same dimensions, with equality if, for example, each entry of $\bm{A}^{(1)}$ and $\bm{A}^{(2)}$ is \emph{not} simultaneously non-zero for both matrices,
       \item $|\bm{A}^{(1)} \bm{A}^{(2)}|\le |\bm{A}^{(1)}||\bm{A}^{(2)}|$ for all square matrices $\bm{A}^{(1)}$ and $\bm{A}^{(2)}$ of the same dimensions,
       \item $|c\bm{A}|\bm{e} = |\bm{A}|(c\bm{e}) = c|\bm{A}|\bm{e}$ for all $c\ge 0$.
\end{itemize}

\begin{lemma}\label{lem:boundPi}
Let $\bm{\Pi}$ be defined through the recursive Equations \ref{eq:rec_aux_1}-\ref{eq:rec_aux_4}, and $\tld{\bm{\Pi}}$  through the recursive Equations \ref{eq:rec_aux_1t}-\ref{eq:rec_aux_4t}. Then, for all $k\ge 0$ and $S\subseteq \{0,1,2,\dots k\}$,
\begin{align}
\sum_{k'\in S} \left|{\bm{\Pi}}(k,k')-\tld{\bm{\Pi}}(k,k')\right|\bm{e}\le k\,C_k\,\bm{e},\label{eq:QtildeQ}
\end{align}
where $C_0=0$ and for $k\ge 1$,
\begin{align}
C_k=\sup\left\{\left(\left|\bm{Q}(\ell,w)-\tld{\bm{Q}}(\ell,w)\right|\bm{e}\right)_j: 0\le w\le \ell\le k-1,\,j\in\mathcal{J}\right\}.\label{eq:defCk}
\end{align}
\end{lemma}
\begin{proof}
We proceed using induction over $k$. It is trivial to verify \eqref{eq:QtildeQ} for $k=0$. Suppose that \eqref{eq:QtildeQ} holds for $k-1\ge 0$, where the sum in the l.h.s. is taken over any subset of $\{0,1,2,\dots k-1\}$. For the induction step at $k$, fix some $S\subseteq\{0,1,2,\dots k\}$ and consider two cases.

\textbf{Case $\{0\}\notin S$.} Here,
\begin{align}
\sum_{k'\in S}&\left|{\bm{\Pi}}(k,k')-\tld{\bm{\Pi}}(k,k')\right|\bm{e}\nonumber\\
&\qquad\qquad + \sum_{k'\in S}\left| \left[{\bm{\Pi}}(k-1,k'-1)-\tld{\bm{\Pi}}(k-1,k'-1)\right]\tld{\bm{Q}}^{\mathrm{d}}(k-1,k'-1)\right|\bm{e}\nonumber\\
&\le\sum_{k'\in S}\left|  {\bm{\Pi}}(k-1,k'-1)\right| \left|{\bm{Q}}^{\mathrm{d}}(k-1,k'-1)-\tld{\bm{Q}}^{\mathrm{d}}(k-1,k'-1)\right|\bm{e}\nonumber\\
&\qquad\qquad + \sum_{k'\in S}\left| {\bm{\Pi}}(k-1,k'-1)-\tld{\bm{\Pi}}(k-1,k'-1)\right|\left|\tld{\bm{Q}}^{\mathrm{d}}(k-1,k'-1)\right|\bm{e}\nonumber\\
&\le C_{k}\sum_{k'\in S}\left|  {\bm{\Pi}}(k-1,k'-1)\right|\bm{e} + \sum_{k'\in S}\left| {\bm{\Pi}}(k-1,k'-1)-\tld{\bm{\Pi}}(k-1,k'-1)\right|\bm{e}\label{eq:Baux1}\\
&\le C_{k}\,\bm{e} + (k-1)\,C_{k-1}\,\bm{e}\le k\,C_k\,\bm{e},\label{eq:Baux2}
\end{align}
where in \eqref{eq:Baux1} we employed 
\[\begin{gathered}
|{\bm{Q}}^{\mathrm{d}}(k-1,k'-1) - \tld{\bm{Q}}^{\mathrm{d}}(k-1,k'-1)|\bm{e}\le C_{k}\,\bm{e}\quad\mbox{and}\\
 |\tld{\bm{Q}}^{\mathrm{d}}(k-1,k'-1)|\bm{e}\le\tld{\bm{Q}}(k-1,k'-1) \bm{e}\le \bm{e},
 \end{gathered}\]
  and in \eqref{eq:Baux2} we used
\begin{align*}
\sum_{k'\in S}(\left|{\bm{\Pi}}(k-1,k'-1)\right|\bm{e})_i& = \sum_{k'\in S}\sum_{j}\Pi_{ij}(k-1,k'-1) \le 1,
\end{align*}
along with the induction step. Thus, \eqref{eq:QtildeQ} follows.

\textbf{Case $\{0\}\in S$.} In this scenario,
\begin{align*}
&\sum_{k'\in S} \left|{\bm{\Pi}}(k,k')-\tld{\bm{\Pi}}(k,k')\right|\bm{e}\\
&=\sum_{k'\in S\backslash\{0\}}  \left|{\bm{\Pi}}(k-1,k'-1){\bm{Q}}^{\mathrm{d}}(k-1,k'-1) -  \tld{\bm{\Pi}}(k-1,k'-1)\tld{\bm{Q}}^{\mathrm{d}}(k-1,k'-1)\right|\bm{e}\\
&\quad + \left|\sum_{k'=0}^{k-1}  {\bm{\Pi}}(k-1,k'){\bm{Q}}^{\mathrm{n}}(k-1,k')-\sum_{k'=0}^{k-1}  \tld{\bm{\Pi}}(k-1,k')\tld{\bm{Q}}^{\mathrm{n}}(k-1,k')\right|\bm{e}\\
&\le\sum_{k'\in S\backslash\{0\}}  \left|{\bm{\Pi}}(k-1,k'-1){\bm{Q}}^{\mathrm{d}}(k-1,k'-1) -  \tld{\bm{\Pi}}(k-1,k'-1)\tld{\bm{Q}}^{\mathrm{d}}(k-1,k'-1)\right|\bm{e}\\
&\quad + \sum_{k'=0}^{k-1} \left|  {\bm{\Pi}}(k-1,k'){\bm{Q}}^{\mathrm{n}}(k-1,k')-\tld{\bm{\Pi}}(k-1,k')\tld{\bm{Q}}^{\mathrm{n}}(k-1,k')\right|\bm{e}\\
&\le\sum_{k'\in S\setminus\{0\}}\left|  {\bm{\Pi}}(k-1,k'-1)\right|\;\left| {\bm{Q}}^{\mathrm{d}}(k-1,k'-1)-\tld{\bm{Q}}^{\mathrm{d}}(k-1,k'-1)\right|\bm{e}\\
&\quad + \sum_{k'\in S\setminus\{0\}}\left| {\bm{\Pi}}(k-1,k'-1)-\tld{\bm{\Pi}}(k-1,k'-1)\right|\;\left|\tld{\bm{Q}}^{\mathrm{d}}(k-1,k'-1)\right|\bm{e}\\
&\quad+\sum_{k''=0}^{k-1}\left|  {\bm{\Pi}}(k-1,k'')\right|\;\left| {\bm{Q}}^{\mathrm{n}}(k-1,k'')-\tld{\bm{Q}}^{\mathrm{n}}(k-1,k''-1)\right|\bm{e}\\
&\quad + \sum_{k''=0}^{k-1}\left| {\bm{\Pi}}(k-1,k'')-\tld{\bm{\Pi}}(k-1,k'')\right|\;\left|\tld{\bm{Q}}^{\mathrm{n}}(k-1,k'')\right|\bm{e},
\end{align*}
where in the last inequality we employed similar steps akin to those taken for the case $\{0\}\notin S$. Since ${\bm{Q}}^{\mathrm{d}}$ and ${\bm{Q}}^{\mathrm{n}}$ have entries which are \emph{not} simultaneously non-zero (ditto $\tld{\bm{Q}}^{\mathrm{d}}$ and $\tld{\bm{Q}}^{\mathrm{n}}$), then the last inequality is equal to
\begin{align}
&\sum_{k''=0}^{k-1}\left|  {\bm{\Pi}}(k-1,k'')\right|\;\left| \left({\bm{Q}}^{\mathrm{d}}(k-1,k'')-\tld{\bm{Q}}^{\mathrm{d}}(k-1,k'')\right)\mathds{1}_{k''+1\in \mathcal{S}} + \left({\bm{Q}}^{\mathrm{n}}(k-1,k'')-\tld{\bm{Q}}^{\ {n}}(k-1,k'')\right)\right|\bm{e}\nonumber\\
&\quad + \sum_{k''\in S\setminus\{0\}}\left| {\bm{\Pi}}(k-1,k'')-\tld{\bm{\Pi}}(k-1,k'')\right|\;\left|\tld{\bm{Q}}^{\mathrm{d}}(k-1,k'')\mathds{1}_{k''+1\in \mathcal{S}} + \tld{\bm{Q}}^{\mathrm{n}}(k-1,k'')\right|\bm{e}\nonumber\\
& \le C_{k}\sum_{k''=0}^{k-1}\left|{\bm{\Pi}}(k-1,k'')\right|\bm{e} + \sum_{k''=0}^{k-1}\left| {\bm{\Pi}}(k-1,k'')-\tld{\bm{\Pi}}(k-1,k'')\right|\bm{e} \label{eq:Baux10} \\
& \le C_{k}\,\bm{e} + (k-1)\,C_{k-1}\,\bm{e}\le k\,C_k\,\bm{e},\label{eq:Baux20}
\end{align}
where in \eqref{eq:Baux10} and \eqref{eq:Baux20} we employed
\[
\begin{gathered}
 \left| ({\bm{Q}}^{\mathrm{d}}(k-1,k'')-\tld{\bm{Q}}^{\mathrm{d}}(k-1,k''))\mathds{1}_{k''+1\in \mathcal{S}} + ({\bm{Q}}^{\mathrm{n}}(k-1,k'')-\tld{\bm{Q}}^{\mathrm{n}}(k-1,k''))\right|\bm{e}\le C_{k}\,\bm{e},\\ 
\left|\tld{\bm{Q}}^{\mathrm{d}}(k-1,k'')\mathds{1}_{k''+1\in \mathcal{S}} + \tld{\bm{Q}}^{\mathrm{n}}(k-1,k'')\right|\bm{e}\le \tld{\bm{Q}}(k-1,k'')\,\bm{e} \le \bm{e},\\
\sum_{k''=0}^{k-1}(\left|{\bm{\Pi}}(k-1,k'')\right|\bm{e})_i \le 1,
\end{gathered}
\]
along with the induction hypothesis. Hence, \eqref{eq:QtildeQ} holds.
\end{proof}
\begin{lemma}\label{lem:boundCk} Let $C_{k}$ be defined by \eqref{eq:defCk} and suppose that the Lipschitz condition \eqref{eq:LipschitzLambda1} holds. Fix $q>1$ and $\varepsilon\in(0,1)$. Then, there exists $\beta(q,\varepsilon)$ such that
\begin{align}
\P\left(C_{\lfloor \gamma^{1+\varepsilon} \rfloor} \ge \beta(q,\varepsilon)\,(\log \gamma)\,\gamma^{-3/2+\varepsilon/2}\right)=o(\gamma^{-q}).\label{eq:Cgammaepsilon0}
\end{align}
\end{lemma}
\begin{proof}
From \eqref{eq:Qbound7},
{
\begin{align}
C_{\lfloor \gamma^{1+\varepsilon} \rfloor} & = \sup\left\{\left(\left|\bm{Q}(\ell,w)-\tld{\bm{Q}}(\ell,w)\right|\bm{e}\right)_j: 0\le w\le \ell\le \lfloor\gamma^{1+\varepsilon} \rfloor-1,\,j\in\mathcal{J}\right\}\nonumber\\
&\le 3K \gamma^{-1} \sup_{\ell'\le {\lfloor \gamma^{1+\varepsilon} \rfloor}}\left|\chi_{\ell}-\tfrac{\ell}{\gamma}\right|.\label{eq:Cgammaepsilon1}
\end{align}
}
In \cite[Lemma 1, Eq (17)]{bladt_peralta} it is proved that
\begin{align}
\P\left(\sup_{\ell'\le {\lfloor \gamma^{1+\varepsilon} \rfloor}}\left|\chi_{\ell}-\tfrac{\ell}{\gamma}\right| \ge \left(2\,e^{1/2+\varepsilon/2+2q}\right)\,(\log \gamma)\,\gamma^{-1/2+\varepsilon/2}\right)=o(\gamma^{-q}),\label{eq:grid-conv1}
\end{align}
which via \eqref{eq:Cgammaepsilon1}, implies \eqref{eq:Cgammaepsilon0} with { $\beta(q,\varepsilon)= 6Ke^{1/2+\varepsilon/2+2q}$}.
\end{proof}
\begin{proof}[Proof of Theorem \ref{th:convergence_densities}.] { For notational convenience, let us write $\ap{p}^{\mathrm{a}}(\cdot)$ and $\ap{p}^{\mathrm{c}}(\cdot)$ to denote the matrices $\{\ap{p}^{\mathrm{a}}_{ij}(\cdot)\}_{i,j\in\mathcal{J}}$ and $\{\ap{p}^{\mathrm{c}}_{ij}(\cdot)\}_{i,j\in\mathcal{J}}$, respectively.} For $s\ge 0$ and { $i\in\mathcal{J}$}, employing the triangle inequality, \eqref{eq:QtildeQ} and $\sum_{\ell \ge 0}\ell\,\mathrm{Poi}_{\gamma s}(\ell)=\gamma s$,
{
\begin{align*}
\left(\left|\ap{p}^{\mathrm{a}}(s) - \apt{p}^{\mathrm{a}}( s)\right|\bm{e}\right)_i &\le \sum_{\ell\ge 0}\left(\left| \ap{p}^{\mathrm{a}}(\ell; s) - \apt{p}^{\mathrm{a}}(\ell; s)\right|\bm{e}\right)_i\\
&\le \sum_{\ell\ge 0}\mathrm{Poi}_{\gamma s}(\ell)\; \left(\left|\bm{\Pi}(\ell, \ell)-\tld{\bm{\Pi}}(\ell, \ell)\right|\bm{e}\right)_i\\
&\le \sum_{0\le\ell\le\lfloor \gamma^{1+\varepsilon} \rfloor}\mathrm{Poi}_{\gamma s}(\ell)\;\left( \left|\bm{\Pi}(\ell, \ell)-\tld{\bm{\Pi}}(\ell, \ell)\right|\bm{e} \right)_i\\
&\qquad + \sum_{\ell\ge \lfloor \gamma^{1+\varepsilon} \rfloor +1 }\mathrm{Poi}_{\gamma s}(\ell)\left( \bm{\Pi}(\ell, \ell)\bm{e} + \tld{\bm{\Pi}}(\ell, \ell)\bm{e}\right)_i\\
&\le \sum_{0\le\ell\le\lfloor \gamma^{1+\varepsilon} \rfloor}\mathrm{Poi}_{\gamma s}(\ell)\,\ell\,C_\ell + 2\sum_{\ell\ge \lfloor \gamma^{1+\varepsilon} \rfloor +1 }\mathrm{Poi}_{\gamma s}(\ell)\\
&\le C_{\lfloor \gamma^{1+\varepsilon} \rfloor}\sum_{0\le\ell\le\lfloor \gamma^{1+\varepsilon} \rfloor}\ell\,\mathrm{Poi}_{\gamma s}(\ell) + 2\sum_{\ell\ge \lfloor \gamma^{1+\varepsilon} \rfloor +1 }\mathrm{Poi}_{\gamma s}(\ell)\\
&\le C_{\lfloor \gamma^{1+\varepsilon} \rfloor}\sum_{\ell \ge 0}\ell\,\mathrm{Poi}_{\gamma s}(\ell) + 2\sum_{\ell\ge \lfloor \gamma^{1+\varepsilon} \rfloor +1 }\mathrm{Poi}_{\gamma s}(\ell)\\
&\le C_{\lfloor \gamma^{1+\varepsilon} \rfloor}(\gamma s)+ 2\sum_{\ell\ge \lfloor \gamma^{1+\varepsilon} \rfloor +1 }\mathrm{Poi}_{\gamma s}(\ell).
\end{align*}
}
{
Similarly, for $0\le v\le s$ and $i\in\mathcal{J}$,
\begin{align*}
&\int_0^v \left(\left| \ap{p}^{\mathrm{c}}(s,v') - \apt{p}^{\mathrm{c}}(s,v')\right| \bm{e}\right)_i\;\dd v' \\
&\quad \le \int_0^s \sum_{\ell\ge 1}\sum_{w=0}^{\ell-1}\mathrm{Erl}_{\ell-w,\gamma}(s-v') \; \mathrm{Poi}_{\gamma v'} (w)\; \left(\left|\bm{\Pi}(\ell, w)-\tld{\bm{\Pi}}(\ell, w)\right|\bm{e}\right)_i\dd v'\\
&\quad =\sum_{\ell\ge 1}\sum_{w=0}^{\ell-1}\left(\int_0^s \mathrm{Erl}_{\ell-w,\gamma}(s-v') \; \mathrm{Poi}_{\gamma v'} (w) \dd v'\right)\; \left(\left|\bm{\Pi}(\ell, w)-\tld{\bm{\Pi}}(\ell, w)\right|\bm{e}\right)_i\\
&\quad =\sum_{\ell\ge 1} \mathrm{Poi}_{\gamma s} (\ell) \sum_{w=0}^{\ell-1} \left(\left|\bm{\Pi}(\ell, w)-\tld{\bm{\Pi}}(\ell, w)\right|\bm{e}\right)_i\\
&\quad \le \sum_{1\le \ell \le \lfloor \gamma^{1+\varepsilon} \rfloor} \mathrm{Poi}_{\gamma s} (\ell) \sum_{w=0}^{\ell-1} \left(\left|\bm{\Pi}(\ell, w)-\tld{\bm{\Pi}}(\ell, w)\right|\bm{e}\right)_i\\
&\qquad\quad + \sum_{\ell\ge \lfloor \gamma^{1+\varepsilon} \rfloor +1 }\mathrm{Poi}_{\gamma s} (\ell)\left(\sum_{w=0}^{\ell-1} (\bm{\Pi}(\ell, w)\bm{e})_i+ \sum_{w=0}^{\ell-1} (\tld{\bm{\Pi}}(\ell, w)\bm{e})_i\right)\\
&\quad \le\sum_{1\le \ell \le \lfloor \gamma^{1+\varepsilon} \rfloor} \mathrm{Poi}_{\gamma s} (\ell) \,\ell\, C_\ell + 2 \sum_{\ell\ge \lfloor \gamma^{1+\varepsilon} \rfloor +1 }\mathrm{Poi}_{\gamma s} (\ell)\\
&\quad = C_{\lfloor \gamma^{1+\varepsilon} \rfloor}(\gamma s) + 2 \sum_{\ell\ge \lfloor \gamma^{1+\varepsilon} \rfloor +1 }\mathrm{Poi}_{\gamma s} (\ell);
\end{align*}
}
note that we employed the identity $\int_0^s \mathrm{Erl}_{\ell-w,\gamma}(s-v') \; \mathrm{Poi}_{\gamma v'} (w) \dd v'=\mathrm{Poi}_{\gamma s}(\ell)$. In short, both { $\sum_{j\in\mathcal{J}}\left|\ap{p}_{ij}^{\mathrm{a}}(s) - \apt{p}_{ii}^{\mathrm{a}}( s)\right|$ and $\int_0^v \sum_{j\in\mathcal{J}}\left| \ap{p}_{ij}^{\mathrm{c}}(s,v') - \apt{p}_{ij}^{\mathrm{c}}(s,v')\right| \dd v'$} are bounded by 
\[C_{\lfloor \gamma^{1+\varepsilon} \rfloor}(\gamma s) + 2 \sum_{\ell\ge \lfloor \gamma^{1+\varepsilon} \rfloor +1 }\mathrm{Poi}_{\gamma s} (\ell)= C_{\lfloor \gamma^{1+\varepsilon} \rfloor}(\gamma s) + 2 \P(\chi_{\lfloor \gamma^{1+\varepsilon} \rfloor +1}<s).\]
The proof can then be concluded by noting that
\begin{align*}
\P&\left( C_{\lfloor \gamma^{1+\varepsilon} \rfloor}(\gamma s) + 2 \P(\chi_{\lfloor \gamma^{1+\varepsilon} \rfloor +1} { \le s)}\ge(s+1)\,\alpha(q,\varepsilon)\,(\log \gamma)\gamma^{-1/2 + \varepsilon/2}) \right)\\
&\le \P\left( C_{\lfloor \gamma^{1+\varepsilon} \rfloor}(\gamma s) \ge s\,\alpha(q,\varepsilon)\,(\log \gamma)\gamma^{-1/2 + \varepsilon/2}) \right)\\
&\quad + \mathds{1}\{2 \P(\chi_{\lfloor \gamma^{1+\varepsilon} \rfloor +1}\le s) \ge \alpha(q,\varepsilon)\,(\log \gamma)\gamma^{-1/2 + \varepsilon/2}\}\\
& = \P\left( C_{\lfloor \gamma^{1+\varepsilon} \rfloor} \ge \alpha(q,\varepsilon)\,(\log \gamma)\gamma^{-3/2 + \varepsilon/2}) \right)\\
&\quad + \mathds{1}\{2 \P(\chi_{\lfloor \gamma^{1+\varepsilon} \rfloor +1}\le s) \ge \alpha(q,\varepsilon)\,(\log \gamma)\gamma^{-1/2 + \varepsilon/2}\} = o(\gamma^{-q}),
\end{align*}
where in the last step we employed \eqref{eq:Cgammaepsilon0} and the fact that $\P(\chi_{\lfloor \gamma^{1+\varepsilon} \rfloor +1}\le s)$ decays at a $q'$-polynomial rate for any $q'>0$ (cf. \eqref{eq:grid-conv1}).
\end{proof}

{
     \section{Proof of Theorem \ref{th:convergence_densities_general}}\label{sec:GeneralProofs}  

As it turns out, the proof of Theorem \ref{th:convergence_densities} in Appendix \ref{sec:convergence_densities} can be translated verbatim to the general state-space case by simply replacing the use of total variation of a matrix with the total variation of a signed kernel. Here we present a sketch of how to update the necessary concepts and their properties.

First, let us note that for all signed kernels $A$, $A^{(1)}$ and $A^{(2)}$ over the measurable space $(\mathcal{J},\mathscr{J})$, the following properties hold:
\begin{itemize}
       \item $|A|(x,\cdot)$ is a nonnegative measure for all $x\in \mathcal{J}$,
       \item $|A^{(1)}+A^{(2)}|(x,B)\le |A^{(1)}|(x,B) + |A^{(2)}|(x,B)$ for all $x\in\mathcal{J}, B\in\mathscr{J}$, with equality holding if, for example, the supports of $A^{(1)}(x,\cdot)$ and $A^{(2)}(x,\cdot)$ are disjoint,
       \item $|A^{(1)}A^{(2)}|(x,B)\le |A^{(1)}||A^{(2)}|(x,B)$ for all $x\in\mathcal{J}, B\in\mathscr{J}$,
       \item $|c A|(x,\dd y) = c|A|(x,\dd y)$ for all $x\in\mathcal{J}, c\ge 0$.
\end{itemize}
With these properties at hand, we can prove the following technical lemma.

\begin{lemma}\label{lem:boundPiGeneral}
For all $k\ge 0$, $x\in\mathcal{J}$ and $S\subseteq \{0,1,2,\dots k\}$,
\begin{align}
\sum_{k'\in S}  \left|{{\Pi}} (k,k')-\tld{{\Pi}} (k,k')\right| (x, \mathcal{J})\le k\,C_k,\label{eq:QtildeQkernel}
\end{align}
where $C_0=0$ and for $k\ge 1$,
\begin{align}
C_k=\sup\left\{\left|{Q}(\ell,w)-\tld{{Q}}(\ell,w)\right|(x,\mathcal{J}): 0\le w\le \ell\le k-1,\,x\in \mathcal{J}\right\}
\end{align}
\end{lemma}
\begin{proof}
As in the proof of Lemma \ref{lem:boundPi}, we employ induction over $k$. Let \eqref{eq:QtildeQkernel} hold for $k-1\ge 0$, and for $S\subseteq\{0,1,2,\dots k\}$, consider two cases.

\textbf{Case $\{0\}\notin S$.} Following similar steps as in Case $\{0\}\notin S$ of Lemma \ref{lem:boundPi},
\begin{align}
&\sum_{k'\in S}\left|{{\Pi}} (k,k')-\tld{{\Pi}} (k,k')\right| (x, \mathcal{J})\nonumber\\
&\le\sum_{k'\in S} \left|  {{\Pi}}(k-1,k'-1)\right| \left|{{Q}}^{\mathrm{d}}(k-1,k'-1)-\tld{{Q}}^{\mathrm{d}}(k-1,k'-1)\right|(x,\mathcal{J})\nonumber\\
&\qquad\qquad + \sum_{k'\in S} \left| {{\Pi}}(k-1,k'-1)-\tld{{\Pi}}(k-1,k'-1)\right|\left|\tld{{Q}}^{\mathrm{d}}(k-1,k'-1)\right|(x,\mathcal{J})\nonumber\\
&\le C_{k}\sum_{k'\in S}  \left|  {{\Pi}}(k-1,k'-1)\right|(x,\mathcal{J}) + \sum_{k'\in S} \left| {{\Pi}}(k-1,k'-1)-\tld{{\Pi}}(k-1,k'-1)\right|\nonumber\\
&\le C_{k} + (k-1)\,C_{k-1}\le k\,C_k,\nonumber
\end{align}
where we employed the induction hypothesis in the second inequality.

\textbf{Case $\{0\}\in S$.} Just as the case $\{0\}\notin S$ is a translation of the one in Lemma \ref{lem:boundPi} (with updated technical language), the current case is also a translation to the one of Lemma \ref{lem:boundPi}. We note that in this case, we employ that the kernels $Q^{\mathrm{d}}$ and  $Q^{\mathrm{n}}$ (ditto $\tld{Q}^{\mathrm{d}}$ and  $\tld{Q}^{\mathrm{n}}$) when evaluated at a fixed point are measures with disjoint support, paralleling the notion of pair of matrices with entries that are not simultaneously non-zero in Lemma \ref{lem:boundPi}. All the other steps are straightforward to generalize.
\end{proof}

We note that, as mentioned earlier, the proof of Lemma \ref{lem:boundPiGeneral} is essentially the same of Lemma \ref{lem:boundPi}, modulo the replacement of total variation of matrices with total variation of kernels. Likewise, employing \cite[Lemma 1, Eq (17)]{bladt_peralta} as in the proof of Lemma \ref{lem:boundCk}, we get in a straightforward manner that for all $q>1$ and $\epsilon>0$ there exists some $\beta'(q,\varepsilon)>0$ such that
\begin{align*}
\P\left(C_{\lfloor \gamma^{1+\varepsilon} \rfloor} \ge \beta'(q,\varepsilon)\,(\log \gamma)\,\gamma^{-3/2+\varepsilon/2}\right)=o(\gamma^{-q}).
\end{align*}
In turn, with steps akin to those in the proof of Theorem \ref{th:convergence_densities}  in Appendix \ref{sec:convergence_densities}, we prove that Theorem \ref{th:convergence_densities_general} holds.
}

\textbf{Acknowledgement.} MB and OP would like to acknowledge financial support from the Swiss National Science Foundation Project 200021\_191984. AM and OP were partially supported by an AXA Research grant on Mitigating risk in the wake of the Covid-19 Pandemic.

\textbf{Declaration} We declare no conflict of interest related to the current manuscript.

\bibliographystyle{apalike}
\bibliography{SemiMarkov_FS_Revised.bib}

\end{document}